\documentclass[11pt]{amsart}
\usepackage[pdftex]{graphicx}
\usepackage{amssymb}
\usepackage{amsmath}
\usepackage{comment}
\usepackage{amsfonts}
\usepackage{amsthm}
\usepackage{amssymb}
\usepackage{mathrsfs}
\usepackage{hyperref}
\usepackage{enumitem}
\usepackage{xcolor}
\newtheorem{theorem}{Theorem}[section]
\newtheorem{proposition}[theorem]{Proposition}
\newtheorem{corollary}[theorem]{Corollary}

\newtheorem{lemma}[theorem]{Lemma}
\newtheorem{notation}[theorem]{Notations}

\theoremstyle{definition}
\newtheorem{definition}[theorem]{Definition}
\theoremstyle{remark}
\newtheorem{remark}[theorem]{Remark}

\newtheorem{example}[theorem]{Example}

\newcommand{\Sp}{\mathbb{S}^2}
\newcommand{\Sbb}{\mathbb{S}}
\newcommand{\bd}{\partial}

\newcommand{\mC}{\mathcal C}

\newcommand{\mP}{\mathcal P}
\newcommand{\mV}{\mathcal V}
\newcommand{\la}{\langle}
\newcommand{\ra}{\rangle}

\newcommand{\eps}{\epsilon}
\newcommand{\R}{\mathbb{R}}

\newcommand{\length}{\mathrm{Length}}
\newcommand{\B}{\mathbb{B}}

\DeclareMathOperator{\area}{\mathrm{Area}}

\numberwithin{equation}{section}

\begin{document}

\title[Steklov transmission eigenvalues]{Optimisation of Steklov transmission eigenvalues and minimal surfaces}
\author[M. Karpukhin]{Mikhail Karpukhin}
\address[Mikhail Karpukhin]{Department of Mathematics, University College London, Gower Street, London WC1E 6BT, UK}
\email{\href{mailto:m.karpukhin@ucl.ac.uk}{m.karpukhin@ucl.ac.uk}}
\author[A. D. Noutchegueme]{Alain Didier Noutchegueme}
\address[Alain Didier Noutchegueme]{Département de Mathématiques et de Statistique, Université de Montréal, CP 6128 Succ Centre-Ville, Montréal, QC H3C 3J7, Canada}
\email{\href{mailto:alain.didier.noutchegueme@umontreal.ca}{alain.didier.noutchegueme@umontreal.ca}}

\date{}
\maketitle
\begin{abstract}
In the present paper, we study the variational properties of Steklov transmission eigenvalues, which can be seen as eigenvalues of the sum of two Dirichlet-to-Neumann operators on two different sides of a given curve contained in a surface. Inspired by the analogous results for Laplacian and Steklov eigenvalues, we show that critical metrics for this problem correspond to the so-called stationary configurations in the Euclidean ball, i.e. pieces of minimal surfaces inside the ball whose normals at the boundary sum up to a vector normal to the boundary sphere. They exhibit strong similarities with free boundary minimal surfaces, and for that reason we call them {\em free curve minimal surfaces}. Furthermore, we study the maximisation problem for these eigenvalues in two contexts. First, in the unconstrained setting, we show that there are no smooth maximal metrics. Second, for the first nontrivial eigenvalue, we explicitly characterize rotationally symmetric maximal metrics on the sphere as those arising from the configurations of stacked catenoids with flat caps.
\end{abstract}
\section{Introduction}

\subsection{Steklov transmission eigenvalues}

It is now a classical result of Nadirashvili~\cite{NadirashviliT2} that critical metrics for the area normalised eigenvalues of the Laplacian arise as metrics on minimal surfaces in round spheres, see also~\cite{ESIextremal, FS}. In recent years, similar correspondences have been observed for other eigenvalue problems, including Steklov eigenvalues~\cite{FS}, Steklov eigenvalues with frequency ~\cite{LM}, functions of Laplace and Steklov eigenvalues~\cite{PetComb1,PetComb2}, and various conformally covariant operators such as Dirac~\cite{KMP}, conformal Laplacian~\cite{GPA} and Paneitz operator~\cite{PA}, see also~\cite{PT} for a unifying approach to this type of problems. We refer to the introduction of~\cite{KKMS} for the up-to-date survey of the most well-studied cases of critical metrics for Laplace and Steklov eigenvalues.

In this paper, we analyze the natural analogue of the classical Steklov problem on closed Riemannian manifold. One of our goals is to show that variational theory for this new problem leads to geometrically interesting objects --- the transmission counterpart of the free-boundary framework in the classical Steklov setting.  Namely, for a closed Riemannian manifold $(M,g)$ and a smooth hypersurface $\Gamma\subset M$ we consider the following eigenvalue problem
\begin{equation}
\label{eq:St_transmission}
\left\{ \begin{array}{rcll}
\Delta_g u & = & 0 & \text{ in }M\setminus\Gamma\\
\partial_{n_1}u+\partial_{n_2}u & = & \tau u & \text{ on }\Gamma,
\end{array}\right.
\end{equation}
where $n_1, n_2$ are two unit normals to $\Gamma$ pointing in the opposite directions. We refer to the problem~\eqref{eq:St_transmission} as {\em Steklov transmission problem} and to numbers $\tau$ that admit a non-trivial solution $u\in C(M)$ to~\eqref{eq:St_transmission} as {\em Steklov transmission eigenvalues}. It is easy to prove (see Section~\ref{sec:eigenv_measure}) that these eigenvalues form a discrete non-negative sequence,
\[
0 =\tau_0(M,\Gamma,g)<\tau_1(M,\Gamma,g)\leq \tau_2(M,\Gamma,g)\leq \cdots \nearrow +\infty ,
\]
where eigenvalues are repeated according to their multiplicity. Steklov transmission eigenvalues are reminiscent of the eigenvalues of the single layer potential operator studied in~\cite{KarpShamma} in the case of a planar curve.

Our main results concern closed surfaces, so in the following we assume that $\dim M = 2$. Define the scale-invariant normalised eigenvalues,
\[
\bar\tau_k(M,\Gamma,g) = \tau_k(M,\Gamma,g)\length_g(\Gamma).
\]  
Our objective is to study the variational theory of functionals $\bar\tau_k(M,\Gamma,g)$ as functions of $g$. In particular, we are interested in the following quantities
\[
T_k(M,\Gamma,[g]):= \sup_{h\in[g]} \bar\tau_k(M,\Gamma,h),\quad T_k(M,\Gamma):=\sup_g \bar\tau_k(M,\Gamma,g),
\]
where $[g] = \{h = e^{2\omega}g,\,\,\omega\in C^\infty(M)\}$ is a fixed conformal class of metrics. Note that the quantity $T_k(M,\Gamma)$ depends only on the diffeomorphism class of the pair $(M,\Gamma)$ rather than the exact embedding of $\Gamma$.

Our first results describe connections between critical points for $\bar\tau_k(M,\Gamma,g)$ and the theory of minimal surfaces and harmonic maps. These results are to a large extent motivated by the now classical theory for the Laplacian~\cite{NadirashviliT2} and Steklov eigenvalues~\cite{FS} that historically formed a mutually beneficial connection between spectral geometry and minimal surface theory.   

\subsection{Free curve minimal surfaces} Given a Riemannian manifold $(Q,h)$ and a map $\Phi\colon (M,g)\to (Q,h)$ its energy is defined as
\[
E(\Phi)=\frac{1}{2}\int _M|d\Phi|_{g,h}^2\,dv_g
\]
 where $|d\Phi|_{g,h}^2$ is the Hilbert-Schmidt norm of $d\Phi$ associated with the metrics $g$ and $h$.

A map $\Phi$ is called {\em harmonic} if it is a critical point of the energy functional. Since we are assuming $\dim M =2$, it is easy to see that $E(\Phi)$ only depends on the conformal class of the metric $g$, hence, harmonicity of $\Phi$ depends only on $[g]$ as well.

If $M$ and $Q$ have non-empty boundary and $\Phi(\bd M)\subset \bd Q$, then $\Phi$ is called {\em free boundary harmonic map} if it is critical for the energy functional in the class of maps sending $\bd M$ to $\bd Q$. Equivalently $\Phi$ is a {\em free boundary harmonic map} if it is harmonic in the interior of $M$ and meets $\bd Q$ at the right angle. 
The theory of free boundary harmonic maps has received a lot of attention in recent years, in no small part due to the connection with the optimisation of Steklov eigenvalues, which manifests itself for $Q = \B^{n+1}\subset \R^{n+1}$, the Euclidean unit ball.

Suppose now that $Q$ still has non-empty boundary, but $M$ is closed with a distinguished fixed curve $\Gamma\subset M$ and $\Phi(\Gamma)\subset \bd Q$. By analogy, we have the following definition.

\begin{definition}
We say that $\Phi\colon (M,\Gamma,g)\to (Q,h)$ is a {\em free curve harmonic map} if $\Phi$ is a critical point of the energy functional in the class of maps satisfying $\Phi(\Gamma)\subset\bd Q$. Equivalently, $\Phi$ is a {\em free curve harmonic map} if $\Phi$ is harmonic on $M\setminus \Gamma$ and $d\Phi(n_1) + d\Phi(n_2)\perp \bd Q$, where $n_1,n_2$ are the two unit normals to $\Gamma$.
\end{definition}
Let us specify to the case $Q = \B^{n+1}\subset \R^{n+1}$, so that one can view $\Phi = (u_1,\ldots, u_{n+1})$ as a vector function. The Euler-Lagrange equation for $\Phi$ then implies that for all $i = 1,\ldots, n+1$ one has
\begin{equation}
\label{eq:FCHM}
\left\{ \begin{array}{rcll}
\Delta_g u_i & = & 0 & \text{ in }M\setminus\Gamma\\
\partial_{n_1}u_i+\partial_{n_2}u_i & = & \tau \rho_{\Phi} u_i & \text{ on }\Gamma,
\end{array}\right.
\end{equation}
where $\rho_\Phi := (\partial_{n_1}\Phi+\partial_{n_2} \Phi,\Phi)> 0$ (as follows from~\cite[Lemma 4.1]{KM}). Let $\hat\rho_\Phi$ be any smooth extension of $\rho_\Phi$ to $M$ and define $g_\Phi = \hat\rho_\Phi^2 g$. Then, $u_i$ are eigenfunctions of the problem~\eqref{eq:St_transmission} on $(M,\Gamma,g_\Phi)$ with eigenvalue $\tau = 1$. In other words, free curve harmonic maps to $\mathbb{B}^{n+1}$ can be seen as maps to the unit ball by Steklov transmission eigenfunctions with eigenvalue $\tau = 1$. 

Our first results concern characterisations of critical points of $\bar\tau_k(M,\Gamma,g)$ as functions of $g$. The rigorous definition of critical metrics for eigenvalue functionals goes back to~\cite{NadirashviliT2} with some variations appearing in~\cite{ESIextremal, FS}, see~\cite{PT} for the most comprehensive treatment. 

\begin{definition}\label{def:extremal}
Let $\mathcal{C}$ be a convex cone of metrics on $M$, and let $g_0 \in \mathcal{C}$ be a metric on $M$. We say that $g_0$ is \textbf{$\mathcal{C}$-critical} for $\bar{\tau}_k$ (or $\bar{\tau}_k$-$\mathcal{C}$-critical) if, for every smooth family of metrics $g(t) \in \mathcal{C}$ with $g(0) = g_0$, one has either 
\[
\bar{\tau}_k(M, \Gamma, g(t)) \leq \bar{\tau}_k(M, \Gamma, g_0) + o(t)
\]
or 
\[
\bar{\tau}_k(M, \Gamma, g(t)) \geq \bar{\tau}_k(M, \Gamma, g_0) + o(t)
\]
as $t \to 0$.

\end{definition}

\begin{notation} 
In this paper, $\mathcal{C}$ will be either the full set of metrics, a conformal class of metrics, or the set of metrics invariant under the action of a compact Lie group $\mathbb{G}$, which in applications is always $SO(2)\approx \mathbb{S}^1$.

\begin{enumerate}
    \item If $\mathcal{C}$ is the full set of metrics, then we say that $g_0$ is \textbf{critical} for $\bar{\tau}_k$ (or $\bar{\tau}_k$-critical).
    \item If $\mathcal{C}$ is a conformal class, then we say that $g_0$ is conformally \textbf{critical} for $\bar{\tau}_k$ (or $\bar{\tau}_k$-conformally critical). 
    \item If $\mathcal{C}$ is the family of metrics invariant under the action of a compact Lie group $\mathbb{G}$ that leaves $\Gamma$ invariant, then we say that $g_0$ is \textbf{critical} for $\bar{\tau}_k$ in the class of $\mathbb{G}$-invariant metrics (or  $\bar{\tau}_k$-critical under the action of $\mathbb{G}$).
\end{enumerate} 
\end{notation}


\color{black}

\begin{theorem}\label{thm:Crithar}
Let $(M,\Gamma)$ be fixed and suppose that $h\in [g]$ is a $\bar\tau_k$-conformally critical metric. Then there exists a free curve harmonic map $\Phi\colon (M,\Gamma,[g])\to \B^{n+1}$ and a function $\omega\in C^\infty(M)$ equal to a constant on $\Gamma$, such that components of $\Phi$ are $\tau_k$-eigenfunctions and $h = e^{2\omega}g_\Phi$. Conversely, if $\Phi\colon (M,\Gamma,[g])\to \B^{n+1}$ is a  free curve harmonic map and $\tau_{k-1}(M,\Gamma,g_\Phi)<\tau_k(M,\Gamma,g_\Phi) = 1$, then the metric $g_\Phi$ is $\bar\tau_k$-conformally critical.
\end{theorem}

We now turn to minimal surfaces. By analogy with free boundary minimal surfaces we introduce the following definition.

\begin{definition}
We say that $\Phi\colon (M,\Gamma,g)\to (Q,h)$ is a {\em free curve minimal immersion} if $\Phi$ 
is a conformal free curve harmonic map. Equivalently, $\Phi(M\setminus \Gamma)$ is a possibly disconnected minimal surface in $Q$, $\Phi(\Gamma)\subset \bd Q$ and the two normals of $\Phi(M\setminus \Gamma)$ to $\Phi(\Gamma)$ sum up to form a vector parallel to $\bd Q$.
We say that $\Phi(M)$ is a free curve minimal surface. 
\end{definition}

\begin{remark}
According to the first variation formula for the area functional, see e.g. \cite{Dierkes_Hildebrant_Sauvigny}, the image $\Phi(M)$ is a stationary point of area under variations of $Q$ that preserve $\bd Q$. In the language on geometric measure theory, $\Phi(M)$ is a stationary integral varifold in $Q$, see e.g.~\cite{CM}.
\end{remark}

Let now $Q = \B^{n+1}\subset \R^{n+1}$, and $\Phi$ be a free curve minimal immersion into $\B^{n+1}$. Then the pullback metric $\Phi^*g_{\R^{n+1}}$ is conformal to $g$ and, hence, the components of $\Phi$ are solution to the equation~\eqref{eq:FCHM} with $g$ replaced by $\Phi^*g$. Furthermore, the density $\rho_\Phi$ can be computed in terms of the angle function $\alpha$ between the normals to $\Phi(M)$ and to the sphere, $(\bd_{n_i} \Phi,\Phi) = \cos \alpha$, so that $\rho_\Phi = 2\cos\alpha$. As a result, $g_\Phi|_{\Gamma}\not = \Phi^*g|_\Gamma$, in contrast with the classical cases of Laplace and Steklov eigenvalues.

\begin{theorem}
\label{thm:CritMin}
Let $(M,\Gamma)$ be fixed and suppose that $g$ is a $\bar\tau_k$-critical metric. Then there exists a free curve minimal immersion $\Phi\colon (M,\Gamma,g)\to \B^{n+1}$ and a function $\omega\in C^\infty(M)$ equal to a constant on $\Gamma$, such that components of $\Phi$ are $\tau_k$-eigenfunctions and $g= e^{2\omega}g_\Phi$. Conversely, if $\Phi\colon (M,\Gamma,g)\to \B^{n+1}$ is a  free curve minimal immersion and $\tau_{k-1}(M,\Gamma,g_\Phi)<\tau_k(M,\Gamma,g_\Phi) = 1$, then the metric $g_\Phi$ is $\bar\tau_k$-critical.
\end{theorem}

\subsection{Upper bounds for \texorpdfstring{$\bar\tau _k$}{tauk}} We explain in Section \ref{sec:eigenv_measure} that $\tau_k(M,\Gamma, g)$ can be seen as eigenvalues associated with a measure supported on $\Gamma$. Thus, it follows from~\cite{GKL, KS} that for any $h\in[g]$
\begin{equation}
\label{ineq:TLbound}
\bar\tau_k(M,\Gamma,h)\leq \Lambda_k(M,[g]),\qquad \bar\tau_k(M,\Gamma,g)\leq \Lambda_k(M),
\end{equation}
where $\Lambda_k$ stands for the suprema of normalised Laplacian eigenvalues $\Delta_g u = \lambda u$ on $(M,g)$, i.e
\[
\Lambda_k(M,[g]) = \sup_{h\in[g]}\lambda_k(M,h)\area(M,h), 
\]
\[
\Lambda_k(M) = \sup_g\lambda_k(M,g)\area(M,g).
\]
Using ideas of homogenisation, \cite{GL,GHL, GKL} we show that bounds~\eqref{ineq:TLbound} can be saturated.
\begin{theorem}
\label{thm:trans_homogenisation}
Let $M$ be a closed surface and $\Gamma\subset M$ be a smooth curve.
\begin{enumerate}
\item For any conformal class $[g]$ on $M$ and for any $k>0$ there exists a sequence of metrics $g_i\in[g]$ and a sequence of curves $\Gamma_i\subset M$ such that $(M,\Gamma_i)$ is in the same homotopy class as $(M,\Gamma)$ and
\[
\lim_{i\to\infty} \bar\tau_k(M,\Gamma_i,g_i) = \Lambda_k(M,[g]).
\]
\item For any $k>0,$ there exists a sequence $g_i$ of metrics such that
\[
\lim_{i\to\infty}\bar\tau_k(M,\Gamma,g_i) = \Lambda_k(M).
\]
\end{enumerate}
\end{theorem}

Furthermore, for $k=1,2$ it was shown in~\cite{KS} that the equality in~\eqref{ineq:TLbound} cannot be achieved. Together with the previous theorem this implies that for a fixed $(M,\Gamma)$ there are no smooth $\bar\tau_k$-maximal metrics. At the same time, it appears likely that the existence theory for $\bar\tau_k$-conformally maximal metric on $(M,\Gamma)$ can be developed along the same lines as in~\cite{KS, Petrides, KNPP} and one can expect the existence of $\bar\tau_1$-conformally maximal metrics for many triples $(M,\Gamma,[g])$.
The absence of $\bar\tau_1$-maximal metrics makes it difficult to construct $\bar\tau_1$-critical points and raises a question whether it is possible to find any non-trivial applications of Theorem~\ref{thm:CritMin}. Our final result gives one such application by considering a particular case of rotationally symmetric metrics on the sphere.

\subsection{Rotational symmetry} More specifically, we set $M=\Sp$, and let  $\Gamma^{(N)}$ be a collection of $N$ parallels, $N\geq 1$, see Figure \ref{Paralells}. We then define
\begin{equation}
\label{def:T1N}
T_1(N):= T^{\Sbb^1}_1(\Sp,\Gamma^{(N)}) := \sup_{g\text{ is $\Sbb^1$-sym.}}\bar\tau_1(\Sp,\Gamma^{(N)}
,g),
\end{equation}
where supremum is over all rotationally symmetric metrics on $\Sp$. As it was observed in~\cite{KKMS}, even though the space of invariant metrics is a proper subspace of the space of all metrics, one still expects that critical points of the restricted functional are global critical points. In Section \ref{sec:Crit_conditions} we prove the following theorem.

\begin{theorem}
    
\label{thm:Crit_S1} 
Suppose that $g_N$ is a  $\bar\tau_1$-critical metric for the problem~\eqref{def:T1N}. Then there exists a rotationally symmetric free curve minimal immersion $\Phi\colon (\Sp,\Gamma^{(N)},g_N)\to \B^{3}$ and a $\Sbb^1-$invariant function $\omega\in C^\infty(\Sp),$ such that components of $\Phi$ are $\tau_1$-eigenfunctions and $g = e^{2\omega}g_\Phi$.

\end{theorem}

As a result, rotationally symmetric critical metrics correspond to free curve minimal surfaces of revolution in $\B^3$. For a given number of $N$ parallels such surfaces then necessarily look like a stack on $(N-1)$ catenoids with flat disk caps on the top and on the bottom. They were studied in~\cite{KZ}, where the authors prove that for each $N\geq1,$ there is a unique (up to a rotation) such free curve minimal surface $\mC(N)$ of revolution. It is then natural to suggest that $\mC(N)$ corresponds to the metric achieving $T_1(N)$. Indeed, we prove the following.

\begin{theorem}
\label{thm:rot_symm}
For each $N\geq1$ there exists a rotationally symmetric metric $g_N$ on $(\Sp,\Gamma^{(N)})$ such that $\bar \tau_1(\Sp,\Gamma^{(N)},g_N) = T_1(N)$. For any other such metric $h$ there exists a function $\omega\in C^\infty(\Sp)$, equal to a constant on $\Gamma^{(N)}$, such that $h = e^{2\omega}g_N$. Furthermore, the quantities $T_1(N)$ possess the following properties:
\begin{enumerate}
\item $T_1(N) = 2\area (\mC(N))$, where $\mC(N)$ is the unique free curve minimal surface of revolution in $\B^3$ whose intersection with $\bd \B^3 = \Sp$ has $N$ connected components;
\item  $\tau_1(\Sp,\Gamma^{(N)},g_N)$ has multiplicity $3$;
\item $\mC(1)$ is a double equatorial disk, and $T_1(1) = 4\pi$;
\item $\mC(2)$ is a "critical drum", see Figure~\ref{Crit_Drum},  and $T_1(2)\approx 4\pi\times 1.56$;
\item $T_1(N)<T_1(N+1)$;
\item $T_1(N)\to 8\pi$ as $N\to\infty$.
\end{enumerate}
\end{theorem}

The case $N=2$ of Theorem~\ref{thm:rot_symm} should be compared with the analysis of rotationally symmetric metrics on the annuli~\cite{FS0}, where it is possible to directly compute the first Steklov eigenvalue for any rotationally symmetric metric and find the maximiser that way. While it is possible to prove Theorem~\ref{thm:rot_symm} in the same way for $N=2$, this method is no longer feasible for larger $N$ as computing Steklov transmission eigenvalues on $(\Sp,\Gamma^{(N)})$ requires solving a polynomial equation of degree $N$. Instead, we prove Theorem~\ref{thm:rot_symm} by using the connection between critical metrics and free curve minimal surfaces. The most technically challenging part is to show the existence of a maximising metric, which ultimately follows from Theorem~\ref{thm:rot_symm}.(5). The maximiser then has to correspond to a free curve surface of revolution in $\B^3$, which is then uniquely identified in~\cite{KZ}. The authors referred to such configuration of catenoid and flat caps as "balanced". These surfaces are obtained by a rotation a particular collection of catenaries and straight lines, see Figure \ref{Catenarydal}.


\begin{figure}[!h]
\hspace*{0cm} \includegraphics[width=11.5cm]{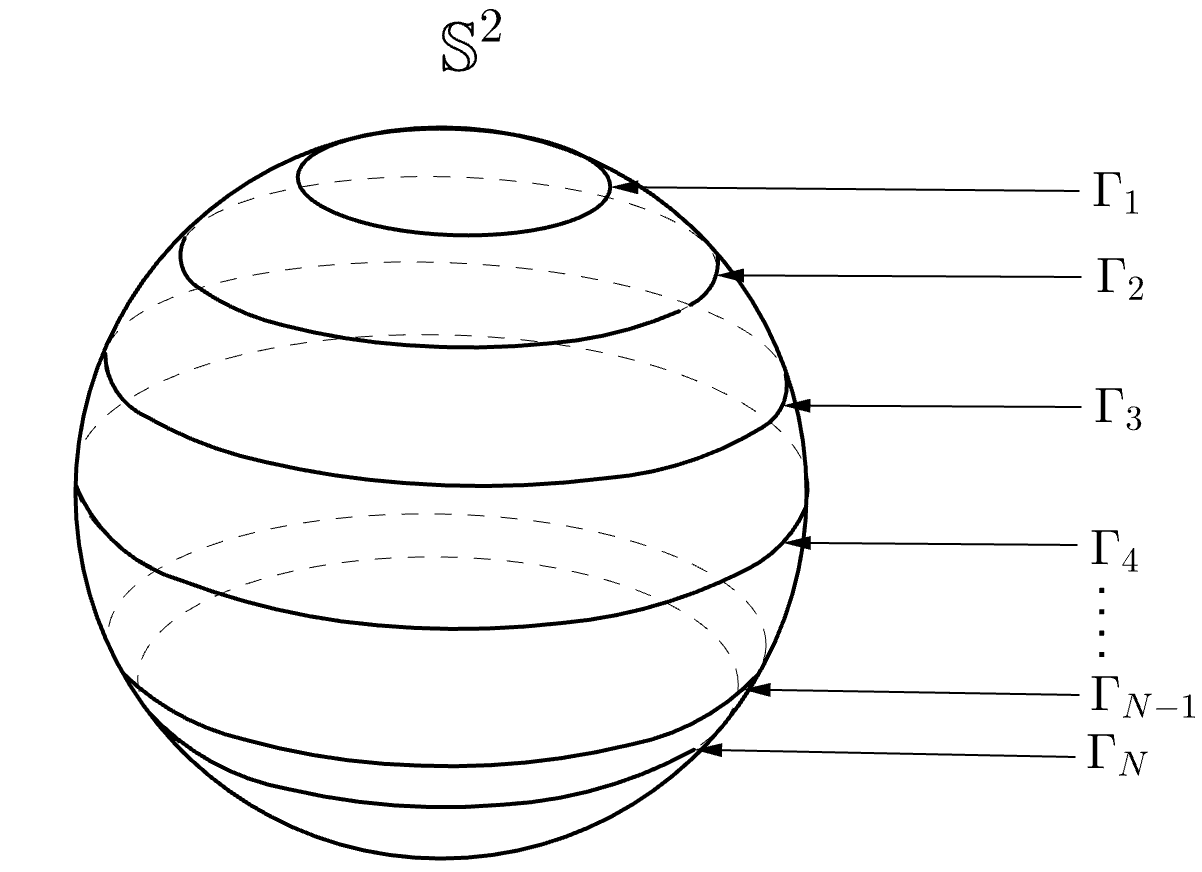}
\caption{$N$ Paralells on $\mathbb{S}^2$}\label{Paralells}

\end{figure}

\begin{figure}[!h]
\hspace*{0.0cm} \includegraphics[width=8.5cm]{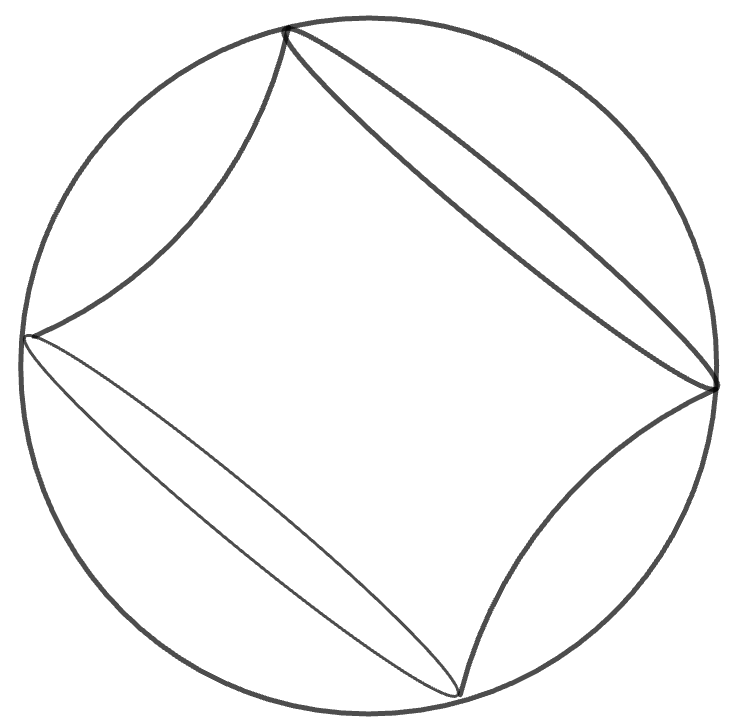}
\caption{The critical drum}\label{Crit_Drum}
\end{figure}


\begin{figure}[!h]
\hspace*{0.0cm} \includegraphics[width=12cm]{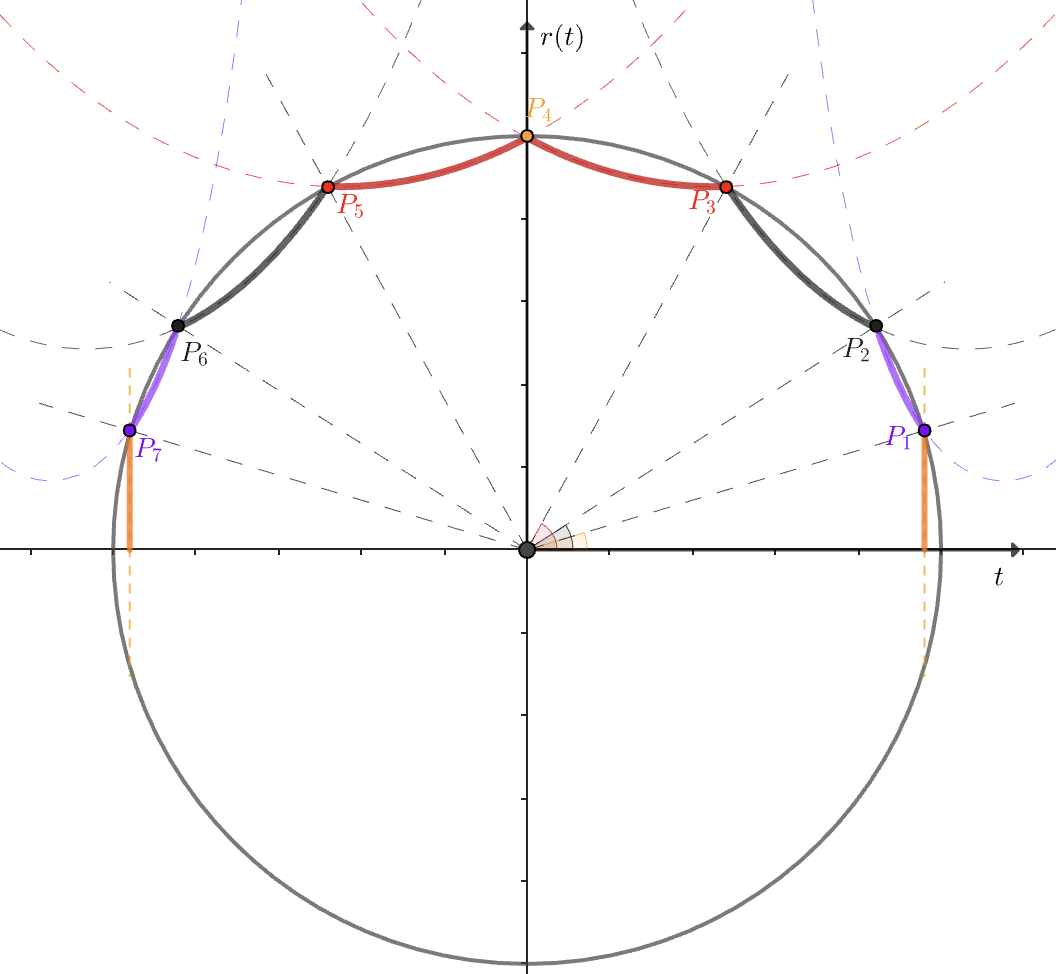}
\caption{$xy-$plane representation of $\mathcal{C}(N)$ before rotation, $N=7$}\label{Catenarydal}
\end{figure}

\subsection{Plan of the paper} 

We start in Section \ref{sec:prel} with preliminary facts about the transmission eigenvalues, including their interpretation as variational eigenvalues of measures and Steklov eigenvalues, multiplicity bound and we recall recent results on homogenisation in spectral geometry.   

Section \ref{sec:Homog} is dedicated to the proof  of Theorem \ref{thm:trans_homogenisation}. 
In Section \ref{sec:Crit_conditions} we establish characterisation of critical metrics in several situations in the spirit of~\cite{FS, NadirashviliT2,KKMS}. 

The aim of section \ref{sec:existence} is to prove the existence of maximal metrics in the class of rotationally symmetric metrics on $\mathbb{S}^2$ with $N$ parallels following the approach in~\cite{KNPP}. This completes the proof of Theorem \ref{thm:rot_symm}

Appendix \ref{app} provides details of the explicit construction of the critical drum. Finally, Appendix \ref{app2} summarises the results from~\cite{KZ} on minimal surfaces of revolution in $\mathbb{B}^3$, which are used in the proof of uniqueness of Theorem~\ref{thm:rot_symm}.(1).

\subsection*{Acknowledgements} The authors would like to thank Iosif Polterovich and Jiahua Zou for fruitful discussions. This project is part of ADN PhD thesis at Université de Montréal under the supervision of Iosif Polterovich and Mikhail Karpukhin. ADN is supported by ISM and the Banque Nationale ESP scholarship.

\section{Preliminaries and Initial results}\label{sec:prel}

\subsection{Eigenvalues of measures}\label{sec:eigenv_measure}
The Steklov transmission eigenvalues~\eqref{eq:St_transmission} fall into the framework of variational eigenvalues~\cite{GKL,GNY,Kokarev}, below we follow the exposition in~\cite{GKL}. Given a closed surface $M$, a conformal class $[g]$ and a Radon measure $\mu$ one defines variational eigenvalues $\lambda_k(M,[g],\mu)$ using the classical Rayleigh quotient as follows 
\[
\lambda_k(M,[g],\mu) = \inf_{F_{k+1}}\sup_{v\in F_{k+1}\setminus\{0\}}\frac{\int_M|du|_g^2\,dv_g}{\int_M u^2\,d\mu},
\]
where $F_{k+1}\subset C^\infty(M)$ is a $(k+1)$-dimensional subspace. One can also define normalised eigenvalues
\[
\bar\lambda_k(M,[g],\mu) = \lambda_k(M,[g],\mu)\mu(M).
\]
Following~\cite{GKL}, we say that the measure $\mu$ is {\em admissible} if the identity operator on $C^\infty(M)$ induces a compact operator $W^{1,2}(M,g)\to L^2(M,\mu)$. The eigenvalues of admissible measures behave similarly to eigenvalues of the classical elliptic operators, e.g. they admit eigenfunctions and have finite multiplicity. Prominent examples include $\mu = dv_h$ for $h\in [g]$, which recovers Laplacian eigenvalues $\lambda_k(M,h)$~\cite[Example 4.2]{GKL}; hypersurface measures $\mu = ds^\Gamma_h$, $h\in[g]$, which yield Steklov transmission eigenvalues $\tau_k(M,\Gamma,h)$~\cite[Example 4.4]{GKL}; and Steklov eigenvalues on manifolds with boundary~\cite[Example 4.3]{GKL}. 

In particular, $\tau_k(M,\Gamma,g)$-eigenfunction $f_k$ satisfies the weak eigenvalue equation,
\[
\int_{M}\nabla f_k\cdot\nabla u \,dv_g= \tau_k(M,\Gamma,g)\int_{\Gamma}f_ku \,ds_g , \ \forall u\in W^{1,2}(M,ds_g^\Gamma).
\]
For the sake of completeness, recall that given a manifold with boundary $(N,g)$ Steklov eigenvalues are given by
 \[
\sigma_k(N,g) = \inf_{F_{k+1}}\sup_{v\in F_{k+1}\setminus\{0\}}\frac{\int_N|du|_g^2\,dv_g}{\int_{\bd N} u^2\,ds_g},
\]
where $ds_g$ is the volume form on $\bd N$.

In~\cite{KS} it is proved that $\bar\lambda_1$-optimisation problem for admissible measures is equivalent to the optimisation of the Laplacian eigenvalues. Namely, for any admissible measure $\mu$
\[
\bar\lambda_1(M,[g],\mu)\leq\Lambda_1(M,[g])
\]
with equality iff $\mu=dv_h$ for $\bar\lambda_1$-maximal metric $h\in[g]$. Hence, $T_1(M,\Gamma,[g])\leq \Lambda_1(M,[g])$. In particular, combining this with the result of Hersch~\cite{Hersch} one has $T_1(\Sp,\Gamma,[g])\leq 8\pi$ for any $\Gamma\subset \Sp$. Similarly, using~\cite[Theorem 5.2]{GKL} together with $\Lambda_k(\Sp) = 8\pi k$ proved in~\cite{KNPP2}, yields the bound 
\begin{equation}
\label{ineq:Tk}
T_k(\Sp,\Gamma)\leq 8\pi k
\end{equation}
valid for any $\Gamma$. This brings us to our first result, which is an improvement of~\eqref{ineq:Tk} in the case $\Gamma$ is the equator.

\begin{proposition}
    Let $M=\Sp$ be a unit sphere in $\R^3$, $g = g_{\mathrm{st}}$ be the standard round metric and $\Gamma$ be the equator. Then for any $h\in[g]$ one has
    \[
    \tau_k(\Sp,\Gamma,h) = 2\sigma_k(\Sp_+,h),
    \]
    where $\Sp_+$ is one of the connected components of $\Sp\setminus\Gamma$.
\end{proposition}
\begin{proof}
By the conformal invariance of the harmonic extension, it is easy to see that by reflecting $\sigma_k(\Sp_+,h)$-eigenfunction across $\Gamma$ we obtain a Steklov transmission eigenfunction with eigenvalue $2\sigma_k(\Sp_+,h)$. Since Steklov eigenfunctions form a basis on $L^2(\Gamma)$, by~\cite[Proposition 4.1]{GKL} there are no other Steklov transmission eigenvalues.
\end{proof}

\begin{corollary}
Let $M=\Sp$ be a unit sphere in $\R^3$, $g = g_{\mathrm{st}}$ be the standard round metric and $\Gamma$ be the equator.
For $(g,\Gamma)$ as above one has
\begin{equation}
\label{eq:St_trans_W}
T_k(\Sp,\Gamma,[g])\leq 4\pi k,
\end{equation}
and the constant on the right hand side is optimal.
\end{corollary}
\begin{proof}
Follows from Hersch-Payne-Weinberger inequality for simply connected surfaces~\cite{GP, HPS}.
\end{proof}

\begin{remark}
Since there exists a conformal automorphism of $\Sp$ that sends equator to any parallel, the sharp upper bound~\eqref{eq:St_trans_W} continues to hold with $\Gamma$ replaced by any circle on $\Sp$. This proves Theorem \ref{thm:rot_symm} for $N=1$.
\end{remark}

\subsection{Equivalent formulations of optimisation problems: negative eigenvalues and densities} The quantities $T_k(M,\Gamma,[g])$ and $T_k(M,\Gamma)$ admit equivalent characterisations that we use interchangeably throughout the paper.
First, note that if $h = \rho^2 g\in [g]$, then $\Delta_{h} = \rho^{-2}\Delta_g$ and the unit normals with respect to both metrics are related as $n_{h,i} = \rho^{-1} n_{g,i}$, so that Steklov transmission eigenvalues on $(M,\Gamma,h)$ satisfy the equation
\begin{equation*}
\left\{ \begin{array}{rcll}
\Delta_{h} u & = & 0 & \text{ in }M\setminus\Gamma\\
\partial_{n_1}u+\partial_{n_2}u & = & \tau \rho u & \text{ on }\Gamma,
\end{array}\right.
\end{equation*}
which coincide with eigenvalues $\lambda_k(M,[g],\rho ds_g^\Gamma)$ associated with the measure $\rho ds_g^\Gamma$ supported on $\Gamma$. In the following we sometimes refer to them as Steklov transmission eigenvalues with density and denote them as $\tau_k(M,\Gamma,g,\rho)$. To summarise, one has
\[
T_k(M,\Gamma,[g]) = \sup_{0<\rho\in C^\infty(\Gamma)}\bar \tau_k(M,\Gamma,g,\rho).
\]
Second, as we have observed earlier, due to the invariance of the Steklov transmission spectrum under the action of diffeomorphisms, $T_k(M,\Gamma)$ only depends on the diffeomorphism class of the pair $(M,\Gamma)$. Namely, we write $(M,\Gamma)\sim (M,\Gamma')$ if there exists a diffeomorphism $F\colon M\to M$ such that $F(\Gamma) = \Gamma'$. With those notations, one has
\[
T_k(M,\Gamma) = \sup_{0<\rho\in C^\infty(M),\,(M,\Gamma')\sim (M,\Gamma)} \bar \tau_k(M,\Gamma',g,\rho).
\]
Finally, we recall the relation between eigenvalue optimisation problems and optimal bounds for the number of negative eigenvalues of the Schr\"odinger operators, see e.g.~\cite{GNY, GNS, KNPP, KS2}. Given $v\in C^\infty(\Gamma)$ consider eigenvalues $\nu_k(M,\Gamma,g,v)$ of the following problem
\begin{equation}\label{eq:neg_eigenvalues}
\left\{ \begin{array}{rcll}
\Delta_g u & = & \nu u & \text{ in }M\setminus\Gamma\\
\partial_{n_1}u+\partial_{n_2}u & = & v u & \text{ on }\Gamma,
\end{array}\right.
\end{equation}
and denote by $N(v)$ the number of negative numbers among eigenvalues $\nu_k(M,\Gamma,g,v)$. We write that $v\in\mP_k(\Gamma)$ if $N(v)\leq k$. By comparing the variational characterizations, it is easy to see that for any $0<\rho\in C^\infty(\Gamma)$ one has that $v_\rho = \tau_k(M,\Gamma,g,\rho)\rho\in \mP_k(\Gamma)$, see e.g.~\cite[Section 2.2]{KNPP} or~\cite[Section 4.2]{KS2}. Furthermore, $\int_\Gamma v_\rho\,ds_g  = \bar\tau_k(M,\Gamma,g,\rho)$. This implies
\[
T_k(M,\Gamma,[g]) = \sup_{0<v\in \mP_k(\Gamma)} \int_\Gamma v\,ds_g \leq \sup_{v\in \mP_k(\Gamma)} \int_\Gamma v\,ds_g=: \mV_k(M,\Gamma,[g]).
\]
The quantity $\mV_k(M,\Gamma,[g])$ is easier to study variationally as one does not need to worry about preserving the condition $v>0$ while choosing a variation. Indeed, one of the approaches to proving existence of conformally maximal metrics is to first prove the equality $T_k = \mV_k$ and then use a less restrictive definition of $\mV_k$. For example, this strategy is used in~\cite{KNPP} in the case of the classical Laplacian. Below, we use similar ideas in the proof of Theorem~\ref{thm:rot_symm}.

\subsection{Convergence and homogenisation}
\label{sec:intro_hom}

Our main inspiration for Theorem~\ref{thm:trans_homogenisation} is an analogous result that has been obtained for Steklov eigenvalues in~\cite{GL}, see also~\cite{GKL}.

\begin{theorem}[Theorem 1.1 in~\cite{GL}]
\label{thm:GL}
    Let $(M,g)$ be a closed Riemannian  surface. Then there exists a sequence of (connected) domains $\Omega_i\subset M$ such that for any $k\geq0$
    \[
    \lim_{i\to\infty}\bar\sigma_k(\Omega_i,g) = \Lambda_k(M,[g]).
    \]
\end{theorem}

Observe that by comparing the variational characterisations of eigenvalues, for $\Omega\subset M$ one always has the following
\begin{equation*}
    \sigma_k(\Omega, g)\leq \tau_k(M,\bd\Omega,g).
\end{equation*}

As a result, Theorem~\ref{thm:GL} combined with~\eqref{ineq:TLbound} implies that for $\Gamma_i:=\bd\Omega_i$ one has
\[
\lim_{i\to\infty}\bar\tau_k(M,\Gamma_i,g) = \Lambda_k(M,[g]).
\]

This is very close to the statement of Theorem~\ref{thm:trans_homogenisation}, the only missing conclusion is the statement about the topological type of the curves $\Gamma_i$. In fact, from the proof of Theorem~\ref{thm:GL} it follows that $\Gamma_i$ is a collection of small circles equidistributed throughout $(M,g)$.  
In Section~\ref{sec:Homog} we show how to deduce Theorem~\ref{thm:trans_homogenisation} from equation~\eqref{eq:Gammai_lim}. We conclude the present section by recalling the following convenient tools for studying limiting behaviour of eigenvalues.
\begin{proposition}[Kokarev~\cite{Kokarev}]
\label{prop:weak_conv}
    If $\mu_n,\mu$ are Radon measures that satisfy $\mu_n\rightharpoonup^*\mu$, then for any $k\geq 0$
    \[
    \limsup_{n\to\infty}\lambda_k(M,[g],\mu_n)\leq\lambda_k(M,[g],\mu)
    \]
\end{proposition}
\begin{theorem}
\label{thm_eigen_conv}
Let $M$ be a closed surface and $[g]$ be a conformal class of metrics on $M$. Suppose that $1\leq p<2$, $\mu$ is an admissible measure on $M$, and let $(\mu_n)_n$ be a sequence of admissible measures on $M$ such that 
\[
\mu_n \longrightarrow \mu \quad \text{in } (W^{1,p}(M))^*.
\]
Then $\mu_n \overset{*}{\rightharpoonup} \mu$ and, for every $k \in \mathbb{N}$, one has
\[
\lambda_k(M, [g], \mu_n) \longrightarrow \lambda_k(M, [g], \mu).
\]
\end{theorem}

\begin{proof}
It is a straightforward consequence of~\cite[Proposition 4.11]{GKL} and of the continuous embedding
\[
W^{1,2,-1/2}(M) \hookrightarrow W^{1,p}(M),
\]
of the Sobolev–Orlicz spaces for $p<2$, see~\cite{GKL} for details.
\end{proof}

\subsection{Multiplicity Bounds} 
\label{sec:multiplicity}

To describe the geometry of minimal immersions associated with the critical metrics of $\bar\tau_1$, it is crucial to understand the maximum possible dimension of the target ball \( \mathbb{B}^{n+1} \) which is controlled by the multiplicity of the corresponding eigenvalue. First, note that one has Courant's Nodal domain theorem for Steklov transmission eigenvalues.
\begin{lemma}\label{Courant_Nodal_Domain}

Let \( (M,g) \) be a closed Riemannian surface and let \( \Gamma \subset M \) be a smooth simple curve, possibly disconnected. Then, any \( \tau_k \)-eigenfunction has at most \( k+1 \) nodal domains
\end{lemma}

\begin{proof}
The proof is identical to the classical Courant nodal theorem for the Laplacian. Let \( u_k \) be a \( \tau_k \)-eigenfunction, and let \( u_j \) be \( \tau_j \)-eigenfunctions for \( j = 0, \dots, k-1 \). Suppose for contradiction that \( u_k \) has at least \( k+2 \) nodal domains \( N_1, \dots, N_{k+2} \). Define \( v_i := u_k \cdot \chi_{N_i} \in H^1(M) \), where \( \chi_{N_i} \) is the characteristic function of \( N_i \), for \( i = 1, \dots, k+1 \). The linear map
\[
F : \mathbb{R}^{k+1} \to \mathbb{R}^k, \quad 
(\alpha_1, \dots, \alpha_{k+1}) \mapsto \left( \int_{\Gamma} v u_0 \, ds_g, \dots, \int_{\Gamma} v u_{k-1} \, ds_g \right),
\]
with \( v = \sum_{i=1}^{k+1} \alpha_i v_i \), must have a non-trivial kernel. Thus, there exists \( v \neq 0 \) with \( \displaystyle\int_{\Gamma} v u_j \, ds_g = 0 \) for all \( j = 0, \dots, k-1 \). Now observe
\[
\int_M |\nabla v|_g^2 \, dv_g 
= \sum_{i=1}^{k+1} \alpha_i^2 \int_{N_i} |\nabla u_k|_g^2 \, dv_g 
= \tau_k \sum_{i=1}^{k+1} \alpha_i^2 \int_{\Gamma \cap N_i} u_k^2 \, ds_g 
= \tau_k \int_{\Gamma} v^2 \, ds_g.
\]
Thus,
\[
\frac{\int_M |\nabla v|^2 \, dv_g}{\int_{\Gamma} v^2 \, ds_g} = \tau_k = \inf_{\substack{w \in H^1(M) \\ \int_{\Gamma} w u_j \, ds_g = 0}} \frac{\int_M |\nabla w|^2 \, dv_g}{\int_{\Gamma} w^2 \, ds_g}.
\]
So \( v \) is a \( \tau_k \)-eigenfunction. But \( v \equiv 0 \) outside \( \overline{\cup_{i=1}^{k+1} N_i} \), which contains the open set \( N_{k+2} \). By unique continuation, this is a contradiction. Hence, \( u_k \) has at most \( k+1 \) nodal domains.    
\end{proof}
Now, one can show that there is an explicit upper bound for the multiplicity of the first non-trivial eigenvalue of the Steklov transmission in the setting of metrics invariant under the action of $\mathbb{S}^1.$

\begin{lemma}\label{Lem_multiplicity}
Let \( \Gamma^{(N)} \) be a union of \( N >1\) parallels on \( \mathbb{S}^2 \). Then for any metric $g$ invariant under the action of \( \mathbb{S}^1 \) by rotations about the \( z \)-axis, the multiplicity of the eigenvalue \( \tau_1(\mathbb{S}^2, \Gamma^{(N)}, g) \) is at most $3$.
\end{lemma}

\begin{proof}

    Since $\mathbb{S}^1$ acts by isometries, the $\tau_1$-eigenspace decomposes into irreducible representations of $\mathbb{S}^1$. These representations are indexed by non-negative integers $j\geq 0$ with $j=0$ corresponding to trivial one-dimensional representations and $j\geq 1$ corresponding to two-dimensional representations, where $\theta\in\mathbb{S}^1$ acts by a rotation by the angle $j\theta$ around the origin. By the separation of variables argument, in cylindrical coordinates $(z,\theta)$ on $\Sp$, these eigenspaces are spanned by $b(z)$ for $j=0$ or $\{a(z)\cos j\theta, a(z)\sin j\theta\}$ for $j\geq 1$. In particular, the function $a(z)\sin j\theta$ contains the curves $\theta = m\frac{\pi}{j}$, $m=0,\ldots, 2j-1$ in its nodal set and, therefore, has at least $4$ nodal domains if $j\geq 2$. Hence, by the Courant nodal domain theorem the only representations that can occur as part of the $\tau_1$-eigenspace correspond to $j=0,1$. To conclude that the multiplicity is at most 3, it suffices to show that vector spaces spanned by $b(z)$ are no more than one dimensional and that vector spaces spanned by $\{a(z)\cos j\theta, a(z)\sin j\theta\}$ are no more than 2-dimensional.

    \begin{enumerate}
        \item Suppose $b_1(z)$ and $b_2(z)$ are two $\tau_1$-eigenfunctions. Let $U\subset \Sp\setminus\Gamma^{(N)}$, $U =\{z_1<z<z_2\}$ be any connected component homeomorphic to an annulus. Since $b_1$ and $b_2$ are affine in $U$, then there is a linear combination $b(z):=ab_1(z)+\tilde{a}b_2(z)$ that is constant in $U$. If $b$ is not zero, then we prove in Appendix \ref{app2} (see Proposition \ref{prop:b_N_monotone}) that $b$ should be strictly monotone. This contradiction implies that $b$ is zero everywhere and thus $b_1(z)$ and $b_2(z)$ span the same vector space.
        \item Suppose that $\{a_1(z)\cos j\theta, a_1(z)\sin j\theta\}$ and $\{a_2(z)\cos j\theta, a_2(z)\sin j\theta\}$ are two 2-dimensional representations of $\tau_1$. Then, by Courant Nodal domain theorem, $a_1$ and $a_2$ should be of constant sign. Take a linear combination $a(z):=ca_1(z)+\tilde{c}a_2(z)$ that vanishes at some point. Then again by Courant Nodal domain theorem, $a(z)$ should vanish identically,  which implies that $a_1$ and $a_2$ are proportional.
    \end{enumerate}
This concludes the proof.

\end{proof}

\section{Homogenisation} \label{sec:Homog}

\subsection{Proof of Theorem~\ref{thm:trans_homogenisation}}

We first recall from Section~\ref{sec:intro_hom} that for $(M,g)$ there exists a sequence $\Gamma_i\subset M$ of collections of small  disjoints circles  such that
\begin{equation}
\label{eq:Gammai_lim}
\lim_{i\to\infty}\bar\tau_k(M,\Gamma_i,g) = \Lambda_k(M,[g]).
\end{equation}

The idea now is to approximate the length measure of $\Gamma_i$ by the measures supported on curves of arbitrary isotopy class
and use Theorem~\ref{thm_eigen_conv}. Namely, the goal of this section is to prove the following proposition, which clearly implies Theorem~\ref{thm:trans_homogenisation}.
\begin{proposition}
    \label{prop:approx_curve}
    For any embedded curve $\Gamma$ there exists a sequence of curves $\Gamma_n'$ isotopic to $\Gamma$ such that for any $p>1$ there exists a density $0<\rho_n\in L^\infty(\Gamma_n')$ satisfying
    \[
    \rho_n\,ds_g^{\Gamma_n'}\to ds_g^{\Gamma_i}\text{ in } (W^{1,p}(M))^*
    \] 
    as $n\to\infty$.
\end{proposition}

This proposition is proved using the following auxiliary lemmas.

\begin{lemma}
\label{prop:hom_2}
 Let $\gamma\subset M$ be a smooth curve on the surface $(M,g)$ and let $ds_g^\gamma$ be the length measure on it. Then $\eps ds_g^\gamma\to 0$ in $(W^{1,p}(M,g))^*$ for any $p>1$.
\end{lemma}

\begin{lemma}
\label{prop:hom_1}
    Let $x\colon[-1,1]\to M$ be a smooth embedded curve on the surface $(M,g)$ and let $\mu_\eps  = ds_g|_{x(-\eps,\eps)}$ be the  length measure on $x(-\eps,\eps)$. Then $\mu_\eps\to 0$ in $(W^{1,p}(M,g))^*$ for any $p>1$.
\end{lemma}

The proofs of both propositions are based on the computation made in~\cite[Section 5.2]{GKL}. Namely, for a curve $\gamma\subset M$ let $N_\delta(\gamma)$ be $\delta$-tubular neighbourhood of $\gamma$. For sufficiently small $\delta>0$, one has $\area (N_\delta(\gamma))\leq C\delta\length(\gamma)$. The following lemma is proved as part of the proof of~\cite[Theorem 5.2]{GKL}, in particular cf.~\cite[inequality~ (5.1)]{GKL} with $\rho \equiv 1$.
\begin{lemma}
    For any $u\in C^\infty(M)$ and $\delta<\delta_0$ one has
    \begin{equation}
    \label{eq:hom}
        \left|\int_\gamma u\,ds_g - \frac{1}{2\delta}\int_{N_\delta(\gamma)}u\,dv_g\right|\leq C \|u\|_{W^{1,1}(N_\delta(\gamma))},
    \end{equation}
    where $\delta_0, C$ depend on $(M,g)$ and the upper bound on the curvature of $\gamma$.
\end{lemma}

\begin{proof}[Proof of Lemma~\ref{prop:hom_2}]
    Set $N_\delta := N_\delta(\gamma)$. First note that for any $p>1$ one has by the H\"older inequality
    \[
     \|u\|_{W^{1,1}(N_\delta)}\leq \|u\|_{W^{1,p}(N_\delta)}\area(N_\delta)^{\frac{1}{q}}\leq C\delta^\frac{1}{q}\|u\|_{W^{1,p}(M)}. 
    \]
    Similarly,
    \[
    \left|\frac{\eps}{2\delta}\int_{N_\delta}u\,dv_g\right|\leq \frac{\eps}{2\delta}\|u\|_{L^p(M)}\area(N_\delta)^{\frac{1}{q}}\leq C\eps\delta^{-\frac{1}{p}}\|u\|_{L^p(M)}.
    \]
    As a result, multiplying~\eqref{eq:hom} by $\eps$ and combining with previous inequalities yields
    \[
    \left|\int_\gamma u\,\eps ds_g\right|\leq  C\left(\eps\delta^{-\frac{1}{p}} +\eps\delta^\frac{1}{q}\right)\|u\|_{W^{1,p}(M)}.
    \]
    Setting $\delta = \eps$ completes the proof.
\end{proof}

\begin{proof}[Proof of Lemma~\ref{prop:hom_1}]
The proof is similar to the previous one. Let $\gamma_\eps:= x(-\eps,\eps)\subset x(-1,1)$, $N_{\delta,\eps}:=N_\delta(\gamma_\eps)$. Then one has $\area(N_{\delta,\eps})\leq C\eps\delta$ and 
    \[
     \|u\|_{W^{1,1}(N_{\delta,\eps})}\leq \|u\|_{W^{1,p}(N_{\delta,\eps})}\area(N_{\delta,\eps})^{\frac{1}{q}}\leq C(\eps\delta)^\frac{1}{q}\|u\|_{W^{1,p}(M)}. 
    \]
Similarly, 
    \[
    \left|\frac{1}{2\delta}\int_{N_{\delta,\eps}}u\,dv_g\right|\leq \frac{1}{2\delta}\|u\|_{L^p(M)}\area(N_{\delta,\eps})^{\frac{1}{q}}\leq C\eps^{\frac{1}{q}}\delta^{-\frac{1}{p}}\|u\|_{L^p(M)}
    \]
Note that the constant in \eqref{eq:hom} is independent of $\varepsilon$. Therefore, combining~\eqref{eq:hom} with previous inequalities yields
\[
\left|\int_{\gamma_\eps} u\, ds_g\right|\leq  C\left(\eps^{\frac{1}{q}}\delta^{-\frac{1}{p}} +(\eps\delta)^\frac{1}{q}\right)\|u\|_{W^{1,p}(M)}
\]
Setting $\delta = \eps^{\frac{p}{2q}}$ completes the proof
\end{proof}

Now we turn to the proof of Proposition~\ref{prop:approx_curve}. We start with the case $\Gamma$ -- connected contractible curve. Let $\gamma_1$ and $\gamma_2$ be two connected components of $\Gamma_i$. Since both of them are small circles and, hence, contractible, there is a curve $\gamma_0$ connecting two points $p_j\in\gamma_j$, $j=1,2$ and not intersecting $\Gamma_i$ in any other points. Let $\eps>0$ small enough so that  the  $\eps$-neighbourhood of $\gamma_0$ is diffeomorphic to a rectangle $(\eps,\eps)\times[-1,1]$ with coordinates $(s,t)$, where $\gamma_0$ corresponds to $s=0$ and $t=(-1)^i$ are two segments $I_{j,\eps}\subset \gamma_j$. Let $\gamma_{0,\eps}$ be the two curves corresponding to $s=\pm\eps$. We then define $\Gamma'$ as
\[
\Gamma':=\Gamma_{ i}\cup\gamma_{0,\eps}\setminus (I_{1,\eps}\cup I_{2,\eps})
\]
and set $\rho_\eps$ to be equal to $1$ on $\Gamma'\cap \Gamma_{i}$ and to $\delta(\eps)$ on $\gamma_{0,\eps}$. Then, one has
\[
\rho_\eps\,ds_g^{\Gamma'} - ds_g^{\Gamma_i} = \delta(\eps)\,ds_g^{\gamma_{0,\eps}} - ds_g^{I_{1,\eps}\cup I_{2,\eps}}
\]
By Lemma~\ref{prop:hom_2} there exists $\delta(\eps)$ such that $(W^{1,p})^*$-norm of the first term on the r.h.s. is smaller than $\eps$, whereas by the second term goes to zero in $(W^{1,p})^*$ by Lemma~\ref{prop:hom_1}. We now note that the resulting curve $\Gamma'$ has one fewer boundary components each of which is still contractible. 
We refer to this procedure as surgery as it is a reminiscent of the notion of surgery used in topology of manifolds. In our context, surgery allows one to decrease the number of connected components while preserving the contractability of each component. More generally, attaching a contractible curve to any other curve using surgery does not change the isotopy class of that curve. To finish the proof of contractible case, simply repeat the surgery procedure until the curve becomes connected.  

Suppose now that $\Gamma$ is arbitrary. Let $\tilde\Gamma$ be the contractible curve and $0<\rho\in L^\infty(\tilde \Gamma)$ be such that 
\[
\|\rho ds_g^{\tilde\Gamma} - ds_g^{\Gamma_i}\|_{(W^{1,p}(M))^*}\leq\eps
\]
Let $\delta$ be small enough so that
\[
\|\delta ds_g^\Gamma\|_{(W^{1,p}(M))^*}\leq\eps
\]
If $\Gamma$ is disjoint from $\tilde\Gamma$, then we can perform a surgery procedure on the measure $\delta ds_g^\Gamma + \rho ds_g^{\tilde\Gamma}$ to obtain a measure supported on the curve isotopic to $\Gamma$ and $3\eps$-close to $ ds_g^{\Gamma_i}$ in $(W^{1,p}(M))^*$-norm, which completes the proof of Proposition~\ref{prop:approx_curve}.

Suppose that $\Gamma$ is not disjoint from $\tilde\Gamma$. The idea then is to replace $\Gamma$ by an isotopic curve that does not intersect $\tilde\Gamma$ and repeat the argument for that curve. Let us show that such isotopic curve exists. Since $\tilde\Gamma$ is connected and contractible, $M\setminus\tilde\Gamma$ has a connected component $D$ diffeomorphic to a disc. Let $\delta>0$ be such that the $\delta$-neighbourhood $U$ of $D$ is still a disc and fix $p\in D$ together with its neighbourhood $p\in V\subset D$. By partition of unity, there exists a diffeomorphism $F$ isotopic to $\mathrm{id}$ of $M$ such that $F=\mathrm{id}$ on $M\setminus U$ and $\bar D\subset F(V)$. Choose $p,V$ such that $V\cap \Gamma = \varnothing$. Then $F(\Gamma)$ is a curve isotopic to $\Gamma$ that is disjoint from $F(V)\supset \bar D$. The proof is complete.

\section{Criticality conditions}\label{sec:Crit_conditions}

 In this section we adapt the definition of critical metrics originated in~\cite{NadirashviliT2} and~\cite{ESIextremal}. Throughout this section, $M$ is a closed surface and $\Gamma$ is a closed curve inside $M$. Our proofs closely follow the approach in~\cite{FS}.

\subsection{Derivative of the eigenvalue}
 Let $l$ be such that $\tau_l(M,\Gamma,g_0)$ (which we will simply denote by $\tau_l(g_0)$) satisfies 
\[
\tau_l(g_0) = \tau_k(M,\Gamma,g_0),
\]
and assume that $\tau_{l-1}(g_0) < \tau_l(g_0)$. Consider a smooth family of metrics $g(t)$ on $M$ such that $g(0) = g_0$ and 
\[
\frac{d}{dt} g(t) = h(t),
\]
where $h(t)$ is a smooth family of symmetric $(0,2)$-tensor fields on $M$. Introduce the following notations
\[
dv_t := dv_{g_t} := dv_{g(t)}, \quad ds_t := ds_{g_t} := ds_{g(t)}, \quad h_t := h(t), \quad \nabla^t := \nabla_{g(t)}
\]
where $\nabla^t$ denotes the gradient operator associated with $g(t)$. 

Let $E_i(g(t)) \subset W^{1,2}(M)$ be the eigenspace associated with the eigenvalue $\tau_i(g(t))$ of $(M,\Gamma,g(t))$ 
\[
\displaystyle 
E_i(g(t)):=\left\{u:\int_{M}\nabla^t u\cdot \nabla^t vdv_{g_t}=\tau_i(t)\int_{\Gamma}uv ds_{g_t}, \forall v\in W^{1,2}(M) \right\}
\]

It is a finite-dimensional vector space (see \cite{GKL}). Define the quadratic form $Q_h$ on $C^\infty(M)$ by:
\[
Q^t_h(u) = -\int_{M} \left\langle \tau(u), h_t \right\rangle dv_{g_t} - \frac{\tau_l(t)}{2} \int_{\Gamma} u^2 h_t(T,T) \, ds_t.
\]
Here, $\tau(u) = du \otimes du - \frac{1}{2} |\nabla^t u|^2 g_t$ is the stress-energy tensor of $u$ with respect to the metric $g(t)$, and $T$ is a unit tangent vector field along $\Gamma$.

\begin{lemma}\label{lem:Derivative}

If $\tau_{l-1}(g_0) < \tau_l(g_0),$ then $\tau_l(t)$ is locally Lipschitz in $t$. Furthermore, if $\dot{\tau_l}(t_0)$ exists, then
\[
\dot{\tau_l}(t_0) = Q^{t_0}_{h_0}(u)
\]
for any $u \in E_l(g(t_0))$ satisfying $||u||_{L^2(\Gamma, g(t_0))} = 1$.

\end{lemma}

\begin{proof} Let $E(t) = \bigoplus_{i=0}^{l-1} E_i(g(t)) \subset L^2(\Gamma, g(t))$. 
Since $\tau_{l-1}(g_0) < \tau_l(g_0)$, it follows that $\tau_{l-1}(g(t)) < \tau_l(g(t))$ for small $t$, ensuring that $E(t)$ has a constant dimension $l$ for sufficiently small $t$. 

Let $P_t: L^2(\Gamma, g(t)) \to E(t)$ be the orthogonal projection onto $E(t)$. To establish the local Lipschitz continuity of $\tau_l(t)$, let $t_1 \neq t_2$ and assume without loss of generality that $\tau_l(t_1) \leq \tau_l(t_2)$. 

Now, let $u$ be an $l$-th eigenfunction of $g(t_1)$ with $\int_{\Gamma} u^2 ds_{t_1} = 1$. Recalling the local expression of the gradient in a coordinate chart:
\[
\nabla_g u = g^{ij} \frac{\partial u}{\partial x_i} \frac{\partial u}{\partial x_j},
\]
it follows from the smoothness of $(g(t))_t$ that
\[
\left| \int_M |\nabla^{t_1} u|^2 dv_{t_1} - \int_M |\nabla^{t_2} u|^2 dv_{t_2} \right| \leq C |t_1 - t_2|
\]
and
\[
\left| \int_{\Gamma} u^2 ds_{t_1} - \int_{\Gamma} u^2 ds_{t_2} \right| \leq C |t_1 - t_2|.
\]
Moreover, we obtain
\[
||P_{t_2}(u)||_{C^1(M \setminus \Gamma)} = ||P_{t_1}(u) - P_{t_2}(u)||_{C^1(M \setminus \Gamma)} \leq C |t_1 - t_2|.
\]
Thus, we deduce
\[
|\tau_l(t_1) - \tau_l(t_2)| = \tau_l(t_2) - \tau_l(t_1) \leq 
\frac{\int_M |\nabla^{t_2} (u - P_{t_2}(u))|^2 dv_{t_2}}{\int_{\Gamma} (u - P_{t_2}(u))^2 ds_{t_2}} 
- \int_M |\nabla^{t_1} u|^2 dv_{t_1}.
\]
Rewriting the second term, we obtain
\[
|\tau_l(t_1) - \tau_l(t_2)| \leq 
\frac{\int_M |\nabla^{t_2} (u - P_{t_2}(u))|^2 dv_{t_2}}{\int_{\Gamma} (u - P_{t_2}(u))^2 ds_{t_2}} 
- \frac{\int_M |\nabla^{t_1} (u - P_{t_1}(u))|^2 dv_{t_1}}{\int_{\Gamma} (u - P_{t_1}(u))^2 ds_{t_1}}.
\]
Finally, we conclude
\[
|\tau_l(t_1) - \tau_l(t_2)| \leq C |t_1 - t_2|.
\]
Thus, $\tau_l(t)$ is locally Lipschitz, completing the first part of the lemma.

Now, choose $u_0 \in E_l(g(t_0))$ and set $u(t) = u_0 - P_t(u_0)$. Define
\[
F(t) = \int_M |\nabla u(t)|^2 dv_t - \tau_l(t) \int_{\Gamma} u(t)^2 ds_t.
\]
Since $F(t) \geq 0$ for all $t$ and $F(t_0) = 0$, it follows that $\dot{F}(t_0) = 0$. Hence, we obtain
\[
\int_M \left[ 2 \left\langle \nabla u_0, \nabla \dot{u_0} \right\rangle - \left\langle du_0 \otimes du_0 - \frac{1}{2} |\nabla u_0|^2 g_0, h_0 \right\rangle \right] dv_{t_0}
\]
\[
= \dot{\tau_l}(t_0) \int_{\Gamma} u_0^2 ds_{t_0} + \tau_l(t_0) \int_{\Gamma} \left(2 u_0 \dot{u_0} + \frac{1}{2} u_0^2 h(T,T) \right) ds_{t_0}.
\]

The computation of the left-hand side can be found in several sources, including \cite{NadirashviliT2, Berger, ESIextremal}. The computation of the right-hand side follows from
\[
\frac{d}{dt} \left( \tau_l(t) \int_{\Gamma} u(t)^2 ds_t \right) = \dot{\tau_l}(t) \int_{\Gamma} u(t)^2 ds_t + \tau_l(t) \left( \int_{\Gamma} 2 u \dot{u} ds_t + \int_{\Gamma} u(t)^2 \frac{d}{dt} (ds_t) \right).
\]
Let $\gamma\colon [0,1] \to M$ be a parametrization of $\Gamma$. Then
\[
\frac{d}{dt} (ds_t) = \frac{d}{dt} \big( |\dot{\gamma}(s)|_{g(t)} \big) ds 
= \frac{d}{dt} \left( \sqrt{g(t)(\dot{\gamma}(s), \dot{\gamma}(s))} \right) ds.
\]
Expanding the derivative, we obtain:
\[
\frac{d}{dt} (ds_t) 
= \frac{1}{2} h(t) \big( \dot{\gamma}(s), \dot{\gamma}(s) \big) |\dot{\gamma}(s)|^{-1} ds.
\]
Rewriting in terms of the unit tangent vector field $T = \frac{\dot{\gamma}(s)}{|\dot{\gamma}(s)|}$ along $\Gamma$, we get
\[
\frac{d}{dt} (ds_t) 
= \frac{1}{2} h(t) \big( T, T \big) |\dot{\gamma}(s)| ds= \frac{1}{2} h(t)(T,T) ds_t.
\]
By taking a partition of unity, we extend this result to the entire $\Gamma$. Since $u_0$ is an $l$-th eigenfunction, we also have
\[
\int_M \left\langle \nabla u_0, \nabla \dot{u_0} \right\rangle dv_{t_0} = \tau_l(t_0) \int_{\Gamma} u_0 \dot{u}_0 ds_{t_0}.
\]
Using this and normalizing $u_0$ such that $||u_0||_{L^2} = 1$, we obtain the final formula:
\[
\dot{\tau_l}(t_0) = -\int_M \left\langle du_0 \otimes du_0 - \frac{1}{2} |\nabla u_0|^2 g_0, h_0 \right\rangle dv_{t_0} 
- \frac{\tau_l(t_0)}{2} \int_{\Gamma} u_0^2 h(T,T) ds_{t_0}.
\]
Equivalently, this can be written as
\[
\dot{\tau_l}(t_0) = Q^{t_0}_{h_0}(u_0).
\]
\end{proof}
Now, consider the Hilbert space
\[
\mathcal{H} = L^2(S^2(M \setminus \Gamma), dv_{g_0}) \times L^2(\Gamma, ds_{g_0}),
\]
which consists of pairs of $L^2$ symmetric $(0,2)$-tensor fields on $M \setminus \Gamma$ and $L^2$ functions on $\Gamma$.

Let us first introduce a notation for a unified treatment of Theorems \ref{thm:Crithar}, \ref{thm:CritMin}, and \ref{thm:Crit_S1}. Let $\mathcal{C}$ be a convex cone of metrics, and let $g_0 \in \mathcal{C}$. We define 
$$T\mathcal{C}_{g_0} = \{(h,f) \in \mathcal{H} \mid \exists (g(t))_t \subset \mathcal{C}, \ g(0) = g_0, \ \dot{g}(0) = h, \ \dot{g}(0)(T,T)|_\Gamma = f \}.$$

\begin{lemma}\label{lem:crit_eigenf} 
Let $\mathcal{C}$ be either the full set of metrics, a conformal class, or the family of metrics invariant under the action of a compact Lie group $G$ that preserves $\Gamma$. If $g_0$ is a $\bar{\tau}_k$-$\mathcal{C}$-critical metric, then for any $(\omega,f) \in \mathcal{H}$ satisfying $\displaystyle\int_{\Gamma} f \, ds_{g_0} = 0$ and $(\omega, \omega(T,T)|_\Gamma) \in T\mathcal{C}_{g_0}$, the following holds.

If there exists $\hat{F} \in L^2(S^2(M \setminus \Gamma))$ such that $(\hat{F}, f) \in T\mathcal{C}_{g_0}$, then there exists $u \in E_k(g_0)$ with $\| u \|_{L^2(\Gamma, g_0)} = 1$ such that 
\[
\left\langle (\omega, \frac{\tau_k(g_0)}{2} f), (\tau(u), u^2) \right\rangle_{L^2} = 0.
\]
\end{lemma}

\begin{proof}
We begin by proving the existence of a sequence of families of metrics $(g_i(t))_t$ such that, for all $i$, we have $g_i(0) = g_0$, $\ell_{g_i(t)}(\Gamma) = \ell_{g_0}(\Gamma)$, and 
\[
(\dot{g}_i(0),\dot{g}_i(0)(T,T)|_\Gamma) \to (\omega,f) \text{ in } \mathcal{H}.
\]

\textit{Claim:} There exists a sequence of families $(\tilde{g}_i(t))_t \subset \mathcal{C}$ and a sequence $(h_i)_i$ with $h_i \in \mathcal{C} \cap C^\infty(S^2(M \setminus \Gamma)) \cap C^0(S^2(M))$ such that, for all $i$,
\[
\dot{\tilde{g}}_i(0) = h_i, \quad (h_i, h_i(T,T)|_\Gamma) \in T\mathcal{C}_{g_0}, \quad \int_{\Gamma} h_i(T,T) ds_{g_0} = 0,
\]
and 
\[
(h_i, h_i(T,T)) \to (\omega, f) \text{ in } \mathcal{H}.
\]

\textbf{Case 1: $\mathcal{C}$ is the full set of metrics.}  
  By density, there exists a sequence $(h_i)_i$ such that for all $i$,  
  \[
  h_i \in C^\infty(S^2(M \setminus \Gamma)) \cap C^0(S^2(M)), \quad \int_{\Gamma} h_i(T,T) ds_{g_0} = 0,
  \]
  and $(h_i, h_i(T,T)) \to (\omega,f)$ in $\mathcal{H}$. Define $\tilde{g}_i(t) := g_0 + t h_i$. Then,  
  \[
  \tilde{g}_i(0) = g_0, \quad \dot{\tilde{g}}_i(0) = h_i, \quad (h_i, h_i(T,T)|_\Gamma) \in T\mathcal{C}_{g_0}.
  \]

\textbf{Case 2: $\mathcal{C}$ is the family of metrics invariant under the action of $\mathbb{G}$.} 
  Taking $h_i$ as the projection of the sequence $(h_i)_i$ onto the space of $\mathbb{G}$-invariant symmetric $(0,2)$-tensor fields, we still have  
  \[
  (h_i, h_i(T,T)) \to (\omega, f) \text{ in } \mathcal{H},
  \]
  since the hypothesis ensures that $\omega$ and $f$ are $\mathbb{G}$-invariant as well. Then, setting $\tilde{g}_i(t) = g_0 + t h_i$ satisfies the required properties.

\textbf{Case 3: $\mathcal{C}$ is a conformal class of metrics.}  
  Consider $(g(t) = e^{\alpha(t)} g_0)_t$, where $(\alpha(t))_t$ is chosen such that $\alpha(0) = 0$ and $\dot{g}(0) = \dot{\alpha}(0) g_0 = \omega$.  
  Let $(\alpha_i(t))_i$ be a sequence such that:
  \[
  \alpha_i(t) \in C^\infty(M \setminus \Gamma) \cap C^0(M), \quad \int_{\Gamma} \dot{\alpha}_i(0) g_0(T,T) ds_{g_0} = 0, \quad \forall i,
  \]
  and 
  \[
  (\dot{\alpha}_i(0) g_0, \dot{\alpha}_i(0) g_0(T,T)|_\Gamma) \to (\omega, f) \text{ in } \mathcal{H}.
  \]
  Setting $\tilde{g}_i(t) = e^{\alpha_i(t)} g_0$ ensures that $(\tilde{g}_i(t))_i$ satisfies the desired properties. This proves the claim. Now, define
\[
g_i(t) = \frac{\ell_{g_0}(\Gamma)}{\ell_{\tilde{g_i}(t)}(\Gamma)} \tilde{g}_i(t).
\]
Then, $g_i(0) = g_0$ and $\ell_{g_i(t)}(\Gamma) = \ell_{g_0}(\Gamma)$. Moreover, since
\[
\frac{d}{dt} \Big|_{t=0} \ell_{\tilde{g}_i(t)}(\Gamma) = \int_{\Gamma} h_i(T,T) ds_{g_0} = 0,
\]
it follows that $\frac{d}{dt} |_{t=0} g_i(t) = h_i$.

Next, let $l$ be such that $\tau_{l-1}(g_0) < \tau_l(g_0) = \tau_k(g_0)$. By criticality, suppose without loss of generality that  
\[
\bar{\tau}_l(g_i(t)) \leq \bar{\tau}_l(g_0) + o(t).
\]
By Lemma \ref{lem:Derivative}, $\dot{\tau}_l(g_i(t))$ exists almost everywhere since $t \mapsto \dot{\tau}_l(g_i(t))$ is Lipschitz.  
Take sequences $(t^i_j)_j$ and $(s^i_j)_j$ such that $t^i_j, s^i_j \to 0$ as $j \to \infty$ and
\[
\dot{\tau}_l(g_i(t^i_j)) \geq 0, \quad \dot{\tau}_l(g_i(s^i_j)) \leq 0.
\]
Choose $u_{i,j}^+ \in E_k(g_i(t^i_j))$ and $u_{i,j}^- \in E_k(g_i(s^i_j))$ such that  
\[
\| u_{i,j}^+ \|_{L^2} = \| u_{i,j}^- \|_{L^2} = 1.
\]
After taking subsequences, we obtain $u_{i,j}^+ \to u^+$ and $u_{i,j}^- \to u^-$ in $C^2(M \setminus \Gamma) \cap C^0(M)$, with $u^+, u^- \in E_k(g_0)$. Then:
\[
\left\langle \left(\omega, \frac{\tau_k(g_0)}{2} f\right), (\tau(u^+), (u^+)^2) \right\rangle_{L^2} 
= \lim_{i,j \to \infty} \left[ - Q^{t^i_j}_{h_i}(u_{i,j}^+) \right] = \lim_{i,j \to \infty} -\dot{\tau}_l(g_i(t^i_j)) \leq 0.
\]
Similarly, we obtain:
\[
\left\langle \left(\omega, \frac{\tau_k(g_0)}{2} f\right), (\tau(u^-), (u^-)^2) \right\rangle_{L^2} 
= \lim_{i,j \to \infty} \left[ - Q^{s^i_j}_{h_i}(u_{i,j}^-) \right] = \lim_{i,j \to \infty} -\dot{\tau}_l(g_i(s^i_j)) \geq 0.
\]
Since the function
\[
t \mapsto \left\langle \left(\omega, \frac{\tau_k(g_0)}{2} f\right), (\tau((1-t) u^- + t u^+), ((1-t) u^- + t u^+)^2) \right\rangle_{L^2}
\]
is continuous on $[0,1]$, there exists $u \in E_k(g_0)$ with $\| u \|_{L^2(\Gamma)} = 1$ such that  
\[
\left\langle \left(\omega, \frac{\tau_k(g_0)}{2} f\right), (\tau(u), u^2) \right\rangle_{L^2} = 0.
\]
This concludes the proof.
\end{proof}

\subsection{Criticality conditions in the space of all \texorpdfstring{$\mathbb G$}{G}-invariant metrics} Let $\mathbb G$ be a compact Lie group acting on $M$ and leaving $\Gamma$ invariant. 
We prove Theorem \ref{thm:Crit_S1} in a slightly more general setting:

\begin{theorem}\label{thm:Crit_S1_2}
Let $g_0$ be a $\mathbb{G}$-invariant metric on $M$ that is critical for $\bar{\tau}_k$ under $\mathbb{G}$. Then, there exist $n>0$, a representation $A\colon \mathbb{G}\to O(n+1)$ and a collection of $\tau_k(g_0)$-eigenfunctions $u_1, \dots, u_{n+1}$ such that $\Phi = (u_1, \dots, u_{n+1})$ defines an equivariant free curve minimal immersion 
\[
\Phi\colon (M, \Gamma) \to \mathbb{B}^{n+1},\quad \Phi(\kappa\cdot x) = A_\kappa(\Phi(x))
\]
such that $g_0 = f_0 g_\Phi$, where $f_0$ constant on $\Gamma$. 

Conversely, if $\Phi \colon (M, \Gamma) \to \mathbb{B}^{n+1}$ is a an equivariant free curve minimal immersion satisfying 
\[
\tau_{k-1}(M, \Gamma, g_\Phi) < \tau_k(M, \Gamma, g_\Phi) = 1,
\]
then the metric $g_\Phi$ is $\bar{\tau}_k$-critical under $\mathbb{G}$.
\end{theorem}

\begin{proof}
Define the convex hull
\[
\mathcal{K} = \operatorname{conv} \left\{ \left(\tau(u), \frac{\tau_k(g_0)}{2} u^2|_\Gamma \right) \mid u \in E_k(g_0) \right\} \subset \mathcal{H}.
\]

\textbf{Claim:} \quad $\left(0, \frac{\tau_k(g_0)}{2} \right) \in \mathcal{K}$.  
Otherwise, by the geometric form of the Hahn-Banach theorem, there exists $(\omega, \xi) \in \mathcal{H}$ such that
\[
\left\langle (\omega, \xi), \left(0, \frac{\tau_k(g_0)}{2} \right) \right\rangle > 0
\]
and
\[
\left\langle (\omega, \xi), \left(\tau(u), \frac{\tau_k(g_0)}{2} u^2 \right) \right\rangle \leq 0, \quad \forall u \in E_k(g_0).
\]

Denote by $L^2_{\mathbb{G}}(S^2(M \setminus \Gamma), g_0)$ the space of $\mathbb{G}$-invariant square-integrable symmetric $(0,2)$-tensor fields on $M \setminus \Gamma$, and let $L^2_{\mathbb{G}}(\Gamma, g_0)$ denote the space of $\mathbb{G}$-invariant functions in $L^2(\Gamma, g_0)$.

Since $g_0$ is $\mathbb{G}$-invariant, the action $R_\kappa$ is an isometry from $(M, g_0)$ to itself for all $\kappa \in \mathbb{G}$, where $R_\kappa(x) = \kappa \cdot x$. Therefore, for all $\kappa \in \mathbb{G}$ and $u \in E_k(g_0)$, we have $u \circ R_{\kappa^{-1}} \in E_k(g_0)$. Thus,
\[
\int_M \left\langle \tau(u \circ R_{\kappa^{-1}}), \omega \right\rangle dv_{g_0} 
+ \frac{\tau_k(g_0)}{2} \int_{\Gamma} (u \circ R_{\kappa^{-1}})^2 \xi \, ds_{g_0} \leq 0.
\]
By change of variables, we obtain:
\[
\int_M \left\langle \tau(u), \omega \circ R_{\kappa} \right\rangle dv_{g_0} 
+ \frac{\tau_k(g_0)}{2} \int_{\Gamma} u^2 (\xi \circ R_{\kappa}) \, ds_{g_0} \leq 0.
\]
Integrating over $\kappa \in \mathbb{G}$ -- equipped with the bi-invariant probability Haar measure $d\mu$ -- and applying Fubini's theorem, we get
\begin{equation} \label{ineq:Sym_Hahn_Banach}
\int_M \left\langle \tau(u), W \right\rangle dv_{g_0} 
+ \frac{\tau_k(g_0)}{2} \int_{\Gamma} u^2 B \, ds_{g_0} \leq 0,
\end{equation}
where
\[
W(x) = \int_{\mathbb{G}} (\omega \circ R_{\kappa})(x) d\mu(\kappa), \quad
B(x) = \int_{\mathbb{G}} (\xi \circ R_{\kappa})(x) d\mu(\kappa).
\]
Similarly, since 
\[
\left\langle (\omega, \xi), \left(0, \frac{\tau_k(g_0)}{2} \right) \right\rangle > 0,
\]
we also obtain
\[
\left\langle (W, B), \left(0, \frac{\tau_k(g_0)}{2} \right) \right\rangle > 0.
\]

Now, define
\[
f = B - \frac{1}{\ell_{g_0}(\Gamma)} \int_{\Gamma} B \, ds_{g_0},
\]
so that
\[
\int_{\Gamma} f \, ds_{g_0} = 0.
\]
Note that $W \in L^2_{\mathbb{G}}(S^2(M \setminus \Gamma), g_0)$ and $f \in L^2_{\mathbb{G}}(\Gamma, g_0)$. The pair $(W, f) \in \mathcal{H}$ satisfies the hypothesis of Lemma \ref{lem:crit_eigenf}, implying the existence of some $u \in E_k(g_0)$ such that
\[
\left\langle (W, \frac{\tau_k(g_0)}{2} f), (\tau(u), u^2) \right\rangle_{L^2} = 0.
\]

However,
\[
\left\langle (W, \frac{\tau_k(g_0)}{2} f), (\tau(u), u^2) \right\rangle_{L^2}
= \int_M \left\langle W, \tau(u) \right\rangle dv_{g_0} 
+ \frac{\tau_k(g_0)}{2} \int_{\Gamma} u^2 f \, ds_{g_0}
\]
\[
= \int_M \left\langle W, \tau(u) \right\rangle dv_{g_0} 
+ \frac{\tau_k(g_0)}{2} \int_{\Gamma} u^2 B \, ds_{g_0} 
- \tau_k(g_0) \frac{\int_{\Gamma} u^2 \, ds_{g_0}}{2 \ell_{g_0}(\Gamma)} \int_{\Gamma} B \, ds_{g_0}
\]
\[
= \left\langle (W, \frac{\tau_k(g_0)}{2} B), (\tau(u), u^2) \right\rangle_{L^2}
- \frac{\int_{\Gamma} u^2 \, ds_{g_0}}{\ell_{g_0}(\Gamma)} 
\left\langle (W, B), \left(0, \frac{\tau_k(g_0)}{2} \right) \right\rangle_{L^2}<0.
\]

This contradiction shows that our assumption was false, and consequently,
\[
\left( 0, \frac{\tau_k(g_0)}{2} \right) \in \mathcal{K},
\]

and there exist $u_1, \dots, u_{n+1} \in E_k(g_0)$ such that
\[
0 = \sum_{i=1}^{n+1} \tau(u_i) = \sum_{i=1}^{n+1} \left( du_i \otimes du_i - \frac{1}{2} |\nabla u_i|^2 g_0 \right) \quad \text{in } M \setminus \Gamma,
\]
and
\[
1 = \sum_{i=1}^{n+1} u_i^2 \quad \text{on } \Gamma.
\]

In $M \setminus \Gamma$, we compute
\[
-\Delta_{g_0} \left( \sum_{i=1}^{n+1} u_i^2 \right) 
= 2 \sum_{i=1}^{n+1} |\nabla u_i|_{g_0}^2 - \sum_{i=1}^{n+1} u_i \Delta_{g_0} u_i
= 2 |\nabla \Phi|^2_{g_0} \geq 0.
\]
Hence, $\sum_{i=1}^{n+1} u_i^2$ is subharmonic in $M \setminus \Gamma$, meaning it cannot attain a local maximum in the open set $M \setminus \Gamma$. Consequently, we deduce:
\[
\sum_{i=1}^{n+1} u_i^2 \leq \sup_{\Gamma} \sum_{i=1}^{n+1} u_i^2 = 1.
\]
This implies that the codomain of $\Phi=(u_1, \dots, u_{n+1})$ is indeed $\mathbb{B}^{n+1}$.

Finally, we observe that
\[
\Phi^*(g_{\text{std}}) = \sum_{i=1}^{n+1} du_i \otimes du_i = \left( \frac{1}{2} \sum_{i=1}^{n+1} |\nabla u_i|^2 \right) g_0,
\]
which implies that $\Phi$ is a conformal free curve harmonic map, or equivalently, 
\[
\Phi \colon (M, \Gamma) \to \mathbb{B}^{n+1}
\]
is a free curve minimal immersion. 

Let us now show that it is possible to choose $\Phi$ to be equivariant. To this end, consider the space $L^2(\mathbb{G},\R^{n+1})$, where $\mathbb{G}$ is once again equipped with the bi-invariant probability Haar measure $d\mu$. Given $\Phi$ we construct the map $\hat\Phi\colon M\to L^2(\mathbb{G},\R^{n+1})$ given by $\left[\hat\Phi(x)\right](\kappa) = \Phi(\kappa\cdot x)$. We first observe that since $\tau_k(M,\Gamma,g)$-eigenspace is finite-dimensional and $\mathbb{G}$-invariant, the image of $\hat\Phi$ lies in a finitely dimensional subspace. If $x\in\Gamma$, then
\[
|\hat\Phi(x)|^2 = \int_{\mathbb{G}}|\Phi(\kappa\cdot x)|^2\,d\mu = 1.
\]
The map $\hat\Phi$ is still by $\tau_k(M,\Gamma,g)$-eigenfunctions and, furthermore, 
\[
\int_\mathbb{G}d[\Phi(\kappa\cdot x)]\odot d[\Phi(\kappa\cdot x)]\,d\mu = \frac{1}{2}\int_{\mathbb{G}}|d[\Phi(\kappa\cdot x)]|^2_{g_0}   \,d\mu = g_\Phi,
\]
where for a map $\Psi$ from $(M,g)$ to an inner product space, we define the bilinear form $(d\Psi\odot d\Psi) (\xi,\eta) =  \la d\Psi(\xi),d\Psi(\eta)\ra$. As a result $\hat\Phi$ is still a free curve minimal immersion by $\tau_k(M,\Gamma,g)$-eigenfunctions. Finally, we observe that $A_\zeta$ given by $(A_\zeta F)(\kappa) = F(\kappa\zeta)$ is an orthogonal transformation of $L^2(\mathbb{G},\R^{n+1})$ and one has \[
[\hat\Phi(\zeta\cdot x)](\kappa) = \Phi (\kappa\zeta\cdot x) =  [\hat\Phi(x)](\kappa\zeta) = [A_\zeta\hat\Phi(x)](\kappa),
\]
so that $\hat \Phi$ is $A$-equivariant.

Conversely, assume that $\Phi = (u_1, \dots, u_{n+1}) \colon (M, \Gamma) \to \mathbb{B}^{n+1}$ is an equivariant free curve minimal immersion satisfying 
\[
\tau_{k-1}(M, \Gamma, g_\Phi) < \tau_k(M, \Gamma, g_\Phi) = 1.
\]
Let $(g(t))_t$ be a smooth variation of metrics with $g(0)=g_\Phi$. We can replace this family by another family $(\tilde{g}(t))_t$ satisfying $\tilde{g}(0) = g_\Phi$ and $\ell_{\tilde{g}(t)}(\Gamma) = \ell_{g_\Phi}(\Gamma)$. The latter condition ensures that
\[
\int_{\Gamma} \dot{\tilde{g}}(t)(T,T) \, ds_{\tilde{g}(t)} = 0.
\]

By Lemma \ref{lem:Derivative}, for almost every $t$, we have
\[
\dot{\tau}_k(t) = Q^{t}_{\dot{\tilde{g}}(t)}(u),
\]
for any $u \in E_k(\tilde{g}(t))$ with $\|u\|_{L^2(\Gamma, \tilde{g}(t))} = 1$.

On the other hand, summing over the eigenfunctions $u_1, \dots, u_{n+1}$, we obtain:
\[
\sum_{i=1}^{n+1} Q^{t}_{\dot{\tilde{g}}(t)}(u_i) = -\int_{M} \left\langle \sum_{i=1}^{n+1} \tau(u_i), \dot{\tilde{g}} \right\rangle dv_{\tilde{g}(t)} -\frac{1}{2} \int_{\Gamma} \left(\sum_{i=1}^{n+1} u_i^2\right) \dot{\tilde{g}}(T,T)ds_{\tilde{g}(t)} = 0.
\]
Thus, there exist $u^1, u^2 \in E_k(g_\Phi)$ such that
\[
\lim_{t \to 0^+} Q^{t}_{\dot{\tilde{g}}(t)}(u^1) \leq 0, 
\quad \text{and} \quad
\lim_{t \to 0^-} Q^{t}_{\dot{\tilde{g}}(t)}(u^2) \geq 0.
\]
Therefore, it follows that 
\[
\bar{\tau}_k(g(t)) \leq \bar{\tau}_k(g_0) + o(t).
\]

\end{proof}

\subsection{Criticality conditions in the space of all metrics} We state theorem \ref{thm:CritMin} in the following equivalent form :

\begin{theorem}\label{thm:CritMin2}
Let $g_0$ be a metric on $M$ that is critical for $\bar{\tau}_k$. Then, there exist $\tau_k(g_0)$-independent eigenfunctions $u_1, \dots, u_{n+1}$ defining a free curve minimal immersion
\[
\Phi = (u_1, \dots, u_{n+1}) \colon (M, \Gamma) \to \mathbb{B}^{n+1},
\]
such that $g_0 = f_0 g_\Phi$ with $f_0$ constant on $\Gamma$. 

Conversely, if $\Phi \colon (M, \Gamma) \to \mathbb{B}^{n+1}$ is a free curve minimal immersion satisfying 
\[
\tau_{k-1}(M, \Gamma, g_\Phi) < \tau_k(M, \Gamma, g_\Phi) = 1,
\]
then the metric $g_\Phi$ is $\bar{\tau}_k$-critical.
\end{theorem}

\begin{proof} The proof is similar and even less constraining than that of Theorem \ref{thm:Crit_S1_2} and is therefore omitted.
\end{proof}

\subsection{Criticality conditions for a fixed conformal class}
We state theorem \ref{thm:Crithar} in the following equivalent form :

\begin{theorem}\label{thm:Crithar2}
Let $g_0$ be a metric on $M$ that is conformally critical for $\bar{\tau}_k$. Then, there exist $\tau_k(g_0)$-independent eigenfunctions $u_1, \dots, u_{n+1}$ defining a free curve harmonic map 
\[
\Phi = (u_1, \dots, u_{n+1}) \colon (M, \Gamma, [g]) \to \mathbb{B}^{n+1},
\]
such that $g_0 = f_0 g_\Phi$ with $f_0$ constant on $\Gamma$. 

Conversely, if $\Phi \colon (M, \Gamma, [g]) \to \mathbb{B}^{n+1}$ is a free curve harmonic map satisfying 
\[
\tau_{k-1}(M, \Gamma, g_\Phi) < \tau_k(M, \Gamma, g_\Phi) = 1,
\]
then the metric $g_\Phi$ is $\bar{\tau}_k$-conformally critical.
\end{theorem}

\begin{proof}
The proof follows similarly to that of Theorem \ref{thm:Crit_S1_2}, with instead
\[
\mathcal{K} = \operatorname{conv} \left\{\left. \frac{\tau_k(g_0)}{2} u^2|_\Gamma \,\right|\,\, u \in E_k(g_0) \right\} \subset L^2(\Gamma),
\]
and is therefore omitted. The converse is similar as well.
\end{proof}

\subsection{Rotationally symmetric case on $\mathbb{S}^2$}

In this section we specify to the setting of Theorem~\ref{thm:rot_symm}, namely to the first non zero eigenvalue on the rotationally symmetric metrics on $\mathbb{S}^2$ with $\Gamma$ consisting of $N$ parallels. 

\begin{theorem}\label{thm:Crit_S1'} 
Let \( \Gamma^{(N)} \) be a union of \( N >1\) parallels on \( \mathbb{S}^2 \). Suppose that \( g_N \) is a \( \bar{\tau}_1 \)-critical metric invariant under the action of \( \mathbb{S}^1 \) by rotations about the \( z \)-axis. Then, for the corresponding equivariant free curve minimal immersion by $\tau_1(\Sp,\Gamma^{(N)},g_N)$-eigenfunctions
\[
\Phi_N \colon (\mathbb{S}^2,\Gamma^{(N)}, g_N) \to \mathbb{B}^{n+1},
\]
one has $n=2$ with $\mathbb{S}^1$ acting on $\R^3$ by the standard rotation around the $z$-axis. 
\end{theorem}

\begin{proof}
First, observe that by the multiplicity bounds of Lemma~\ref{Lem_multiplicity} one has that $n\leq 2$. Let us now show $n=2$. 

    Since the immersion $\Phi_N$ is equivariant, its coordinate functions span an invariant subspace of the $\tau_1$-eigenspace and, hence, is also decomposed into irreducible representations. Suppose that $n=0$, then the only possibility is that the coordinate function of $\Phi_N$ is $\mathbb{S}^1$-invariant and, hence, it can not be conformal as $d\Phi_N(\bd_\theta)\equiv 0$. If $n=1$, then either the coordinate functions form a two-dimensional $\mathbb{S}^1$-invariant subspace -- which is impossible for the same reason -- or the coordinate functions span the space $\{a(z)\cos \theta, a(z)\sin \theta\}$. Since $\Phi_N$ is equivariant, we in fact have that up to a rotation in $\R^2$ one has $\Phi_N = (a(z)\cos \theta, a(z)\sin \theta)$. Let $U\subset \Sp\setminus\Gamma^{(N)}$, $U =\{z_1<z<z_2\}$ be any connected component homeomorphic to an annulus. One has $a(z_1) = a(z_2) = \pm 1$, hence, there exists $z_0$ such that $a'(z_0) = 0$ and, in particular, for $z=z_0$ one has $d\Phi_N(\bd_z) = 0$, contradicting conformality of $\Phi_N$.

    Thus, we conclude that $n=2$, in which case the coordinates of $\Phi_N$ span $\tau_1$-eigenspace, and the only possible decomposition of this $3$-dimensional space into irreducible representations of $\mathbb{S}^1$ is one copy of trivial representation and one copy of two-dimensional representation corresponding to $j=1$. Indeed, it cannot be decomposed into three copies of a trivial representation, as otherwise $\Phi_N$ could not be  conformal.
    As a result, up to a rotation of $\R^3$ one has
    \[
    \Phi_N(z,\theta) = (a_N(z)\cos\theta, a_N(z) \sin\theta, b_N(z)),
    \]
     which is the standard representation of $\mathbb{S}^1$ on $\R^3$ by rotations around an axis.
\end{proof}

\begin{corollary}
\label{cor:FCMS_uniqueness} Let \( \Gamma^{(N)} \) be a union of \( N >1\) parallels on \( \mathbb{S}^2 \). Suppose that \( g_N \) is a \( \bar{\tau}_1 \)-critical metric invariant under the action of \( \mathbb{S}^1 \) by rotations about the \( z \)-axis. Then, the corresponding map $\Phi_N$ is injective and 
the image $\Phi_N(\Sp)$ is a unique (up to an orthogonal transformation of $\R^3$) collection of disks and catenoids with boundaries on $\Sp$ meeting at equal and opposite angles at their common boundaries.
One has  
\[
\lim_{N\to\infty}\mathrm{Area}(\Phi_N(\Sp))= 4\pi.  
\] 
Furthermore, for each $N\geq 1$, for any other $\bar{\tau}_1$-critical metric $h$ invariant under the action of \( \mathbb{S}^1 \) by rotations about the \( z \)-axis, there exists a function $\omega\in C^\infty(\Sp)$, equal to a constant on $\Gamma^{(N)}$, such that $h = e^{2\omega}g_N$. 
\end{corollary}\color{black} 
\begin{proof}
The first assertion follows from the fact that plane and catenoid are the only rotationally symmetric minimal surfaces in $\R^3$. The angle condition is an immediate consequence of conformality and the Steklov transmission eigenvalue equation $(\bd_{n_1}+\bd_{n_2})\Phi_N = \tau_1\Phi_N$.
    
The uniqueness and the area convergence follow from the results of~\cite{KZ}, where the authors investigated the so-called balanced configurations of catenoids, which are exactly the objects described in the first part of the Corollary. In Appendix~\ref{app2} we summarize the arguments of~\cite{KZ} and show how they imply the uniqueness of $\Phi_N(\Sp)$. The uniqueness of the image 
together with Theorem~\ref{thm:CritMin2} completes the characterisation of all $\bar\tau_1$-critical metrics.
\end{proof}

\section{Existence of Rotationally Symmetric Maximizers} \label{sec:existence}

In this section we prove the existence statement of Theorem~\ref{thm:rot_symm}. In order to do that it will be convenient for us to transport the problem to the cylinder $\mathbb{S}^1\times\R$, which is to be conformally equivalent to $\Sp\setminus\{N,S\}$, where $N,S$ are the north and south pole respectively, via the map $\pi\colon\mathbb{S}^1\times\R\to\Sp\setminus\{N,S\}$  given by
\[
\pi(\theta,z) = \pi_{st}^{-1}(e^{z+2\pi i\theta} ),
\]
where $\pi_{st}\colon \Sp\to\mathbb{C}$ is the stereographic projection. Since $\pi$ is conformal, it preserves the Dirichlet integral and, therefore, we can express $\tau_1(\Sp,\Gamma,\mu)$ in terms of coordinates on $\mathbb{S}^1\times\R$.
In these new coordinates
\[
\pi^{-1}(\Gamma) = \bigcup_{k=1}^N \Gamma_k, \quad \Gamma_k := \{z=z_k\} \subset \mathbb{S}^1\times\R,
\]
where $z_i$ are increasing. Since the vertical translations $z\mapsto z+z_0$ are conformally invariant, this configuration is characterised only in terms of vertical spacings between consecutive curves. To that end we define \( \alpha = (\alpha_1, \dots, \alpha_{N-1}) \in (0, \infty)^{N-1} \) by $\alpha_i:=z_{i+1}-z_i$. Let \( \beta = (\beta_1, \dots, \beta_N) \in \Delta^{N-1}\subset(0, \infty)^N \), $\sum_{i=1}^N\beta_i = 1$ be the weight of $\pi^*\mu$ on $\Gamma_i$.

Note that $C_0^\infty(\Sp\setminus\{N,S\})$ is dense in $H^1(\Sp)$, hence, for $(\alpha,\beta)$ as before, the eigenvalue of the measure $\pi^*\mu$ equals
\[
\mathcal{F}_N(\alpha, \beta) := 
\inf_{\substack{u \in \mathcal{C}_0^\infty(\mathbb{S}^1 \times \mathbb{R}) \\ u \neq 0,\ \int_\Gamma u \rho\, ds_{g_{0}} = 0}} 
\frac{\int_{\mathbb{S}^1 \times \mathbb{R}} |\nabla u|^2_{g_{0}} \, dv_{g_0}}{\int_\Gamma u^2 \rho \, ds_{g_{0}}},
\]
where $\rho=\beta_i$ on $\Gamma_i$. Note that this expression makes sense even if some of $\beta_i$ vanish, so it makes sense to consider $\mathcal F_N$ as a function on $\R_{>0}^{N-1}\times \overline{\Delta}^{N-1}$.

In particular, one has
\[
T_1(N) = \sup_{\alpha,\beta}\mathcal{F}_N(\alpha,\beta)
\]
and we aim to show that the supremum is achieved for some $(\alpha,\beta)\in \R_{>0}^{N-1}\times \Delta^{N-1}$.
The proof is split in several steps. First, we show that for each fixed $\alpha$ the quantity 
\[
T_1(N,\alpha) := \sup_{\beta}\mathcal{F}_N(\alpha,\beta)
\]
is achieved and for any maximizing $\beta$ one has $\beta_i>0$ for all $i=1,\ldots, N$. Then we study $T_1(N,\alpha)$ as a function of $\alpha$ and, in particular show that for any sequence $\alpha_i$ escaping to the boundary of $(0, \infty)^{N-1}$ one has that $\limsup\limits_{i\to\infty}T_1(N,\alpha)\leq T_1(N-1)$. Finally, we combine these two facts in an inductive argument to finish the proof of Theorem~\ref{thm:rot_symm}. Overall, the strategy of proof is analogous to the strategy for the proof of existence of maximizers for other eigenvalue functionals (cf.~\cite{Petrides1}), although it is much less technically demanding in our case.

\subsection{Maximizers of $T_1(N,\alpha)$} 

The goal of this section is to prove the following theorem.
\begin{theorem}
    \label{thm:conf_exist}
    For any \( N \in \mathbb{N}^* \) and any $\alpha\in (0,+\infty)^{N-1}$ there exists $\beta\in (0,+\infty)^N$ such that 
    \[
    T_1(N,\alpha) = \mathcal F(\alpha,\beta).
    \]
    Furthermore, if $\beta'\in [0,+\infty)^N$ is such that $\beta'_i = 0$ for some index $i$, then $\mathcal F_N(\alpha,\beta')<T_1(N,\alpha).$
\end{theorem}

We follow the strategy in~\cite{KNPP}, i.e. we interpret $T_1(N,\alpha)$ as a solution of an optimization problem for Schr\"odinger-type operators.

Throughout this section $\alpha$ is fixed and let $v\in C^\infty(\Gamma).$ 
For any integer $k \in \mathbb{N}$, define the $k$-th variational eigenvalue:
\[
\nu_k(\alpha,v) := \inf_{F_{k+1}} \sup_{0 \neq u \in F_{k+1}} 
\frac{
\int_{M} |\nabla u|^2_{g_0} \, dv_{g_{0}} - \int_{\Gamma} v u^2 \, ds_{g_{0}}
}{
\int_{M} u^2 \, dv_{g_{0}}
},
\]
where the infimum is taken over all $(k+1)$-dimensional subspaces $F_{k+1} \subset W^{1,2}(M, g_{0})$.
These eigenvalues satisfy the standard min-max principle:
\[
\nu_0(\alpha,v) < \nu_1(\alpha,v) \leq \nu_2(\alpha,v) \leq \cdots \nearrow +\infty.
\]
These are the eigenvalues of the following transmission-type spectral problem
\[
\left\{
\begin{array}{rcll}
\Delta_{g_{0}} u_k &=& \nu_k(\alpha,v)\, u_k & \text{in } M \setminus \Gamma, \\
\partial_{n_1} u_k + \partial_{n_2} u_k &=& v\, u_k & \text{on } \Gamma.
\end{array}
\right.
\]
In the following we consider rotationally symmetric potentials $v\equiv v_i\in\R$ on $\Gamma_i$ and identify $v$ with an element of $\R^N$.
Define the admissible set
\[
\mathcal{P}_k := \left\{ v \in \mathbb{R}^N : \nu_k(\alpha,v) \geq 0 \right\}
\]
and let 
\[
\mathcal{V}_k := \sup_{v \in \mathcal{P}_k} \int_{\Gamma} v \, ds_{g_{0}}
\]
\begin{definition}
Let $v \in \mathbb{R}^N$. We say that $v$ is a \emph{critical potential} for the functional
\begin{equation}
\label{def:F}
F(v) = \int_{\Gamma} v \, ds_{g_{0}}
\end{equation}
within the admissible set $\mathcal{P}_k$, if for every smooth variation $v(t) \in \mathcal{P}_k$, defined for $t \in [0, \varepsilon)$ with $v(0) = v$ one has 
\[
\int_{\Gamma} v(t) \, ds_{g_{0}} \leq \int_{\Gamma} v \, ds_{g_{0}} + o(t) \quad \text{as } t \to 0^+.
\]
\end{definition}
\begin{proposition}\label{prop:neg_critical}
Let $v \in \mathcal{P}_1$ be a critical potential for the functional $F(v)$ in $\mathcal{P}_1$. Then, there exists a free curve harmonic map 
\[
\Phi := (u_1, \dots, u_{n+1})\colon M \to \mathbb{B}^{n+1}
\]
by $\nu_1$-eigenfunctions such that
\[
\sum_{i=1}^{n+1} u_i^2 = 1 \quad \text{on } \Gamma, \quad \text{and} \quad \nu_1(v, \gamma) = 0.
\]
In particular, $v>0$ on $\Gamma$.
\end{proposition}

\begin{proof}
 The proof parallels that of Theorem~\ref{thm:Crit_S1_2}, with key differences highlighted. We begin by establishing that $\nu_1(v, \gamma) = 0$. Consider the variation $v(t) = v + t$. If $\nu_1(\alpha,v) > 0$, then by continuity, $\nu_1(\alpha, v(t)) > 0$ for small $t > 0$. However, the variation of the functional satisfies
\[
\int_\Gamma v(t) \, ds_{g_{0}} = \int_\Gamma v \, ds_{g_{0}} + t \cdot \ell_{g_{0}}(\Gamma) > \int_\Gamma v \, ds_{g_{0}} + o(t),
\]
which contradicts the criticality of $v$. Thus, $\nu_1(\alpha,v) = 0$.

For any smooth deformation $v(t)$, the map $t \mapsto \nu_1(\alpha,v(t))$ is Lipschitz (as in Lemma~\ref{lem:Derivative}). When differentiable, we have the standard variational formula
\[
\dot{\nu}_1(v(t)) = -\int_\Gamma \dot{v}(t) u_t^2 \, ds_{g_{0}},
\]
where $u_t$ is a corresponding eigenfunction normalized by $\|u_t\|_{L^2(M)} = 1$.

Let $E_1$ be the $\nu_1(\alpha,v)$-eigenspace, we claim that $1 \in \mathcal{K} := \mathrm{conv}\{ u^2|_\Gamma : u \in E_1 \}$. Suppose, for contradiction, that $1 \notin \mathcal{K}$. Then by the Hahn–Banach theorem, there exists $\xi \in L^2(\Gamma, g_{0})$ such that
\[
\int_\Gamma \xi \, ds_{g_{0}} > 0, \quad \text{and} \quad \int_\Gamma \xi u^2 \, ds_{g_{0}} \leq 0 \quad \forall u \in E_1.
\]

Define the rotationally invariant function $B(z) := \int_{\mathbb{S}^1} \xi(\theta, z) \, d\theta$, and note that
\[
\int_\Gamma B \, ds_{g_{0}} > 0, \quad \int_\Gamma B u^2 \, ds_{g_{0}} \leq 0 \quad \forall u \in E_1.
\]
Set $f := B - \frac{1}{\ell_{g_{0}}(\Gamma)} \int_\Gamma B \, ds_{g_{0}}$, which is again $\mathbb{S}^1$-invariant and satisfies $\int_\Gamma f \, ds_{g_{0}} = 0$. Then
\[
\int_\Gamma f u^2 \, ds_{g_{0}} = \int_\Gamma B u^2 \, ds_{g_{0}} - \frac{1}{\ell_{g_{0}}(\Gamma)} \int_\Gamma u^2 \, ds_{g_{0}} \int_\Gamma B \, ds_{g_{0}} < 0.
\]

Now consider the deformation $v(t) := v + t(f + \delta)$ for small $\delta > 0$. When differentiable, we get:
\[
\dot{\nu}_1(v(t)) = -\int_\Gamma (f + \delta) u_t^2 \, ds_{g_{0}} = -\int_\Gamma f u_t^2 \, ds_{g_{0}} - \delta \int_\Gamma u_t^2 \, ds_{g_{0}}>0.
\]

By compactness and smooth convergence $u_t \to u \in E_1(\Gamma, v)$ in $C^2$, and the above inequality implies $\dot{\nu}_1(v(t)) > 0$ for sufficiently small $t$, yielding
\[
\nu_1(v(t)) > \nu_1(v) = 0.
\]
Thus, $v(t) \in \mathcal{P}_1$, but
\[
\int_\Gamma v(t) \, ds_{g_{0}} = \int_\Gamma v \, ds_{g_{\text{std}}} + t \delta \cdot \ell_{g_{0}}(\Gamma) > \int_\Gamma v \, ds_{g_{0}} + o(t),
\]
contradicting the criticality of $v$. Hence, the claim $1 \in \mathcal{K}$ holds.

It remains to prove that $v > 0$ on $\Gamma$. Recall that the eigenfunctions satisfy the jump condition:
\[
\partial_{n_1} u_i + \partial_{n_2} u_i = v u_i \quad \text{on } \Gamma.
\]
Multiplying by $u_i$ and summing, we obtain:
\[
\sum_{i=1}^{n+1} u_i (\partial_{n_1} u_i + \partial_{n_2} u_i) = v \sum_{i=1}^{n+1} u_i^2 = v.
\]
Since $\sum_{i=1}^{n+1} u_i^2$ is subharmonic in $M \setminus \Gamma$, and equals $1$ on $\Gamma$, the Hopf maximum principle implies:
\[
\sum_{i=1}^{n+1} u_i \partial_{n_1} u_i = \frac{1}{2} \partial_{n_1} \left( \sum_{i=1}^{n+1} u_i^2 \right) > 0, \quad 
\sum_{i=1}^{n+1} u_i \partial_{n_2} u_i = \frac{1}{2} \partial_{n_2} \left( \sum_{i=1}^{n+1} u_i^2 \right) > 0.
\]
Thus, $v > 0$ on $\Gamma$, concluding the proof.
\end{proof}

 For $M > 0$, define the truncated quantities
\[
\mathcal{V}_1(M) := \sup_{v \in \mathcal{P}_1 \cap [-M, M]^N} \int_{\Gamma} v \, ds_{g_{0}}.
\]
Note that since $\mathcal{P}_1$ is closed, the set $\mathcal{P}_1 \cap [-M, M]^N$ is compact. By continuity of the linear functional $v \mapsto \int_{\Gamma} v \, ds_{g_{0}} = 2\pi \sum_{i=1}^N v_i$, it follows that there exists a maximizer $v(M) \in \mathcal{P}_1 \cap [-M, M]^N$ satisfying
\[
\int_{\Gamma^\gamma} v(M) \, ds_{g_{0}} = \mathcal{V}_1(M).
\]

\begin{proposition}
\label{prop:upper_bound_neg_egv}
There exists a constant $K > 0$ such that for any $v \in \mathcal{P}_1$ and for all $i \in \{1, \dots, N\}$, we have
\[
v_i \leq K.
\]
\end{proposition}

\begin{proof}
We start from the inequality that is valid for any $v\in\mathcal P_1$
\[
0 \leq \nu_1(\alpha,v) = \inf_{\substack{V \subset W^{1,2}(M) \\ \dim V = 2}} \sup_{\substack{u \in V \\ u \neq 0}} \frac{\int_{M} |\nabla u|^2_{g_{0}} \, dv_{g_{0}} - \int_{\Gamma} v u^2 \, ds_{g_{0}}}{\int_{M} u^2 \, dv_{g_{0}}}.
\]

Fix $i \in \{1, \dots, N\}$, and let $u_1^i$, $u_2^i$ be two test-functions with the following properties
\begin{itemize}
  \item[(1)] $u_1^i = 1$ and $u_2^i = \cos \theta$ on $\Gamma_i$;
  \item[(2)] $u_1^i$ and $u_2^i$ vanish identically on all other $\Gamma_j$ for $j \neq i$.
\end{itemize}

Let $V = \mathrm{span}\{u_1^i, u_2^i\}$, and let $u^i = a u_1^i + b u_2^i \in V$ be such that
\[
\frac{\int_{M} |\nabla u^i|^2_{g_{0}} \, dv_{g_{0}} - \int_{\Gamma} v (u^i)^2 \, ds_{g_{0}}}{\int_{M} (u^i)^2 \, dv_{0}} = \sup_{\substack{u \in V \\ u \neq 0}} \frac{\int_{M} |\nabla u|^2_{g_{0}} \, dv_{g_{0}} - \int_{\Gamma^\gamma} v u^2 \, ds_{g_0}}{\int_{M} u^2 \, dv_{g_{0}}}.
\]

We now estimate both sides of this inequality. Since $u_1^i$ and $u_2^i$ are supported away from  $\Gamma_j$ for $j \neq i$, we have
\[
\int_{\Gamma_i} v_i (a + b \cos \theta)^2 \, d\theta \leq \int_{M} |\nabla u^i|^2_{g_{0}} \, dv_{g_{0}} \leq (a^2 + b^2) K_i,
\]
for some constant $K_i > 0$ depending on the energy of $u_1^i$ and $u_2^i$ (by Cauchy–Schwarz and bilinearity). On the other hand,
\[
\int_{\Gamma_i} v_i (a + b \cos \theta)^2 \, d\theta = v_i \int_0^{2\pi} \left( a^2 + 2ab \cos \theta + b^2 \cos^2 \theta \right) \, d\theta = v_i (2\pi a^2 + \pi b^2).
\]

Combining the two expressions, we get
\[
v_i (2\pi a^2 + \pi b^2) \leq K_i (a^2 + b^2),
\]
which implies
\[
v_i \leq \frac{K_i (a^2 + b^2)}{2\pi a^2 + \pi b^2} \leq K_i',
\]
for some constant $K_i' > 0$ independent of $v$.
Finally, define $K := \max\{K_i' : i = 1, \dots, N\}$, which gives the desired uniform upper bound on each component $v_i$ for all $v \in \mathcal{P}_1$.
\end{proof}

\begin{proof}[Proof of Theorem~\ref{thm:conf_exist}]
We first argue that the functional $F(v)$ defined by~\eqref{def:F} achieves its maximum on $\mathcal P_1$ and, hence, for the maximizing potential $v_{\max}$ one has $v_{\max}>0$ by Proposition~\ref{prop:neg_critical}. 

By Proposition~\ref{prop:upper_bound_neg_egv}, there exists a constant \( K > 0 \) such that the maximizing potentials \( v(M) \in \mathcal{P}_1 \cap [-M, M]^N \) satisfy \( v_i(M) \leq M \) for all \( M \in \mathbb{N} \) and all \( i = 1, \ldots, N \).
Now consider the sequence \( \mathcal{V}_1(M) = \int_{\Gamma^\gamma} v(M) \, ds_{g_{0}} \). Since the potentials are uniformly bounded from above, we may apply the reverse Fatou lemma to deduce
\[
\mathcal{V}_1 = \lim_{M \to \infty} \mathcal{V}_1(M) = \limsup_{M \to \infty} \int_{\Gamma^\gamma} v(M) \, ds_{g_{0}} 
\leq \int_{\Gamma} \limsup_{M \to \infty} v(M) \, ds_{g_{0}}.
\]
Define \( v_{\max} := \limsup_{M \to \infty} v(M) \). By passing to a suitable subsequence, one has \( v(M) \to v_{\max}\) component-wise. Moreover, as \( \mathcal{V}_1 \geq 0 \), each component \( v_{\max,i} > -\infty \). Since each \( v(M) \in \mathcal{P}_1 \), and this set is closed, we deduce \( v_{\max} \in \mathcal{P}_1 \), and therefore, $\int_{\Gamma} v_{\max} \, ds_{g_{0}} \leq \mathcal{V}_1$. It follows that
\[
\int_{\Gamma} v_{\max} \, ds_{g_{0}} = \mathcal{V}_1.
\]
Furthermore, by Proposition~\ref{prop:neg_critical} one has $v_{\max}>0$ and any $v\in\mathcal P_1$ with non-positive components can not be maximizing, so that if $v_i\leq 0$ for some $i=1,\ldots, N$, then $F(v)<\mathcal{V}_1$.

Now, for each $0\not\equiv\beta\in [0,+\infty)^N$ consider $v_\beta:= \tau_1(M,\Gamma,\beta)\beta$. Then, it is easy to see that $\nu_1(\alpha,v_\beta) = 0$ so that $v_\beta\in \mathcal P_1$ and $F(v_\beta) = \mathcal F_N(\alpha,\beta)$. In particular, one has that $T_1(N,\alpha)\leq \mathcal V_1$. Conversely, if $v\geq 0$, $v\not\equiv 0$, then $\tau_1(M,\Gamma,v) = 1$ and $\mathcal F_N(\alpha,v) = F(v)$. Since $v_{\max}>0$ one has
\[
\mathcal F_N(\alpha, v_{\max}) = F(v_{\max}) =\mathcal V_1 \geq F(v) = \mathcal F_N(\alpha, v),
\]
where the middle inequality is strict if $v_i= 0$ for some $i$. This completes the proof.
\end{proof}

\subsection{$T_1(N,\alpha)$ as a function of $\alpha$} In this section we take the maximizing sequence $\alpha^j\in(0,+\infty)^{N-1}$ such that
\[
T_1(N,\alpha^j)\to T_1(N).
\]
Let us start with the following auxiliary lemma.
\begin{lemma}
\label{lem:conv_1}
    If $t^j_i\to t_i$, and $\beta_i^j\to\beta_i$ for $i=1,\ldots,N$ then
    \begin{equation}
    \label{eq:lem_conv_1}
    \sum_{i=1}^N \beta_i^jds_{g_0}^{\{z=t^j_i\}} \to \sum_{i=1}^N \beta_ids_{g_0}^{\{z=t_i\}}
     \end{equation}
    in $(W^{1,p}(M))^*$ for all $p>1$ and weakly-$*$.
\end{lemma}
\begin{proof}
    Let $\mu_j$ the measure in the l.h.s. of~\eqref{eq:lem_conv_1} and $\mu$ be the measure in the r.h.s. Let $\mu_j' = \sum_{i=1}^N\beta^j_i\,ds^{\Gamma_i}_{g}$, then
    \[
    \|\mu_j - \mu\|\leq \|\mu_j-\mu_j'\| + \|\mu_j'-\mu\|.
    \]
    Since $\beta^j\to\beta$, Lemma~\ref{prop:hom_2} implies that the second term goes to $0$. For the first term, one has
    \[
    \|\mu_j-\mu'_j\| \leq \sum_{i=1}^N |\beta_i^j| \| ds_g^{\{z=t^j_i\}} - ds_g^{\{z=t_i\}}\|.
    \]
    At the same time, by the fundamental theorem of calculus, for any smooth $f$ one has
    \[   
    \left|\int_{z=t^j_i}f\,ds_{g_0} - \int_{z=t_i}f\,ds_{g_0}\right| = \left|\int_{[t^j_i,t_i]\times\mathbb{S}^1}\partial_z f\,dv_{g_0}\right|\leq (2\pi|t_i-t^j_i|)^{\frac{1}{q}}\|f\|_{W^{1,p}(M)}.
    \]
    Since $|\beta_i|\leq 1$, the lemma follows.
    
    To prove the weak convergence observe that for any continuous $v$ one has
    \[
     \left|\int_{z=t_j}f\,ds_{g_0} - \int_{z=t}f\,ds_{g_0}\right| \leq  \int_{z=t}|f(\theta, t) - f(\theta,t_j)|\,ds_{g_0}\to 0
    \]
    by the uniform continuity of continuous functions on bounded cylinders.
\end{proof}

\begin{proposition}
\label{prop:alphasup}
    Suppose that there exists $i=1,\dots,N-1$ such that $\limsup\alpha^j_i = +\infty$. Then $T_1(N)\leq T_1(N-1)$.
\end{proposition}
\begin{proof}
Up to a choice of a subsequence, we may assume that $\alpha^j_i\to+\infty$. Furthermore, since vertical shifts are conformal, we may assume that $\Gamma^j_{i+1} = \{z = \frac{1}{2}\alpha^j_i\}$, $\Gamma^j_{i} = \{z = -\frac{1}{2}\alpha^j_i\}$. Define 
\[
B^j_1 = \sum_{k=1}^i\beta^j_k,\qquad B^j_2 = 1-B^j_1 = \sum_{k=i+1}^N\beta^j_k
\]

Suppose that  $\liminf B^j_1> 0$ and $\liminf B^j_2> 0$. Consider a function $u_j$ depending only on $z$ such that $u_j(z) = -B_1^j$ for $z\geq \frac{1}{2}\alpha^j_i$, $u_j(z) = B_2^j$ for $z\leq -\frac{1}{2}\alpha^j_i$ and $u(z)$ is linear between $\pm\frac{1}{2}\alpha^j_i$. Then $u_j$ is an admissible test-function for $\mathcal F_{N}(\alpha^j,\beta^j)$ and, in particular,
\[
\mathcal F_{N}(\alpha^j,\beta^j)\leq \frac{1}{(\alpha^j_i)^2B_1^jB_2^j(B_1^j+B_2^j)}
\]
Since $B_1^j$ and $B_2^j$ are bounded from below away from zero, we conclude that $\mathcal F_{N}(\alpha^j,\beta^j)\to 0$, which is a contradiction.

This argument implies that as soon as $\alpha^j_i\to+\infty$ for some $i$, the weights $\beta^j_k$ have to converge to zero either for all $k\leq i$ or for all $k\geq i+1$. Applying the argument several times, we conclude that there exists $1\leq i_1\leq i_2\leq N$ such that 
\begin{enumerate}
    \item $i_2-i_1+1\leq N-1$;
    \item $\beta_i^j\to 0$ for all $i<i_1$ and $i>i_2$; 
    \item $\alpha_i^j\to\infty$ for $i=i_1-1$ and $i=i_2$;
    \item $\alpha_i^j$ are bounded independently of $j$ for $i_1\leq i \leq i_2-1$,
\end{enumerate}
where in all these statements we assume $\alpha_i^j=\infty$ for $i\ne 1,\ldots,N-1$ and $\beta_i^j=0$ for $i\ne 1,\ldots, N$.
Furthermore, condition (4) implies that, up to another vertical shift, the measure
\[
\tilde\mu^j = \sum_{i=i_1}^{i_2}\beta_i^j\,ds^{\Gamma_i^j}_{g_0}
\]
is supported in a fixed compact part of $\mathbb{S}^1\times\R$. Up to a choice of a subsequence, we may then assume that 
\[
\tilde\mu^j\rightharpoonup^*\tilde\mu = \sum_{i=i_1}^{i_2}\beta_i\,ds^{\Gamma_i}_{g_0}
\]
and, in particular, $\tilde\mu$ is a measure supported on at most $N-1$ parallels. By condition (2), the measure 
\[
\mu_j' = \tilde\mu^j - \sum_{i=1}^N\beta_i^j\,ds_{g_0}^{\Gamma^j_i}
\]
converges weakly-$*$ to 0. Therefore, by Lemma~\ref{lem:conv_1} and Proposition~\ref{prop:weak_conv} we have 
\[
\lim_{j\to\infty} F_N(\alpha^j,\beta^j)\leq F_{i_2-i_1+1}(\{\alpha_{i_1},\ldots,\alpha_{i_2-1}\},\{\beta_{i_1},\dots,\beta_{i_2}\})\leq T_1(N-1)
\]
and the proof is complete.
\end{proof}

\begin{proposition}
\label{prop:alphainf}
    Suppose that there exists $i=1,\dots,N-1$ such that $\liminf\alpha_j^i = 0$. Then $T_1(N)\leq T_1(N-1)$.
\end{proposition}
\begin{proof}
By Proposition~\ref{prop:alphasup}, we can assume without loss of generality that all $\alpha^j_i$ are uniformly bounded and, up to a choice of a subsequence $\lim\alpha_j^i = 0$ as $j\to\infty$. Thus, up to a vertical shift and possibly further choice of a subsequence, Lemma~\ref{lem:conv_1} implies that we may assume the following convergence of measures in $(W^{1,p}(M))^*$
\[
\mu_j :=\sum_{i=1}^N\beta_i^j\,ds_{g_0}^{\Gamma^j_i}\to \sum_{i=1}^N\beta_i\,ds_{g_0}^{\Gamma_i} =:\mu,
\]
where parallels $\Gamma_i$ are limits of parallels $\Gamma^j_i$. Since $\lim\alpha_j^i = 0$, the measure $\mu$ has at most $N-1$ distinct parallels and the proposition follows form the continuity of eigenvalues statement, see Theorem~\ref{thm_eigen_conv}. 
\end{proof}

\begin{corollary}
    Suppose that $T_1(N)>T_1(N-1)$. Then there exists $\alpha\in(0,+\infty)^{N-1}$, $\beta\in\Delta^{N-1}$ such that
    \[
    \mathcal F_N(\alpha,\beta) = T_1(N).
    \]
\end{corollary}

\begin{proof}
    Let $\alpha^j$ be a maximizing sequence as before and let $\beta^j\in\Delta^{N-1}$ be such that $T_1(N,\alpha^j) = \mathcal F_N(\alpha^j,\beta^j)$. Then Propositions~\ref{prop:alphasup} and~\ref{prop:alphainf} together with the assumption $T_1(N)>T_1(N-1)$ imply that for each $i$ the sequence $\{\alpha_i^j\}_j$ is bounded away from $0$ and $+\infty$. Therefore, up to a choice of a subsequence we can assume that $\alpha^j\to \alpha\in (0,+\infty)^{N-1}$ and $\beta^j\to\beta\in\overline{\Delta}^{N-1}$.
    By Lemma~\ref{lem:conv_1} and Theorem~\ref{thm_eigen_conv} this implies
    \[
     T_1(N,\alpha^j) =\mathcal F_N(\alpha^j,\beta^j) \to \mathcal F_N(\alpha,\beta).
    \]

    Finally, observe that $\mathcal F_N(\alpha,\beta) = T_1(N)$ also implies $\mathcal F_N(\alpha,\beta) = T_1(N,\alpha)$. Thus, Theorem~\ref{thm:conf_exist} implies $\beta_i>0$ for all $i$.
\end{proof}
\subsection{Existence of maximizers for $T_1(N)$}
\begin{theorem}
For any $N\geq 2$ one has $T_1(N)>T_1(N-1)$. In particular, there exist $\alpha\in (0,+\infty)^{N-1}$ and $\beta\in \Delta^{N-1}$ such that $\mathcal F_N(\alpha,\beta) = T_1(N)$.
\end{theorem}
\begin{proof}
    The proof is by induction on $N$. Since $T_1(1)=2\pi$ and the value of $\bar\tau_1$ on the critical drum exceeds $2\pi$, one has $T_1(2)>T_1(1)$. This covers the base of the induction. Suppose that $T_1(N)>T_1(N-1)$, we now prove that $T_1(N+1)>T_1(N)$. Since $T_1(N)>T_1(N-1)$, there exists $\alpha,\beta$ such that $\mathcal F_N(\alpha,\beta) = T_1(N)$. Let $\alpha_{N}>0$ and consider the sequences $\alpha' = \{\alpha_1,\ldots,\alpha_{N-1},\alpha_N\}\in(0,+\infty)$, $\beta' = \{\beta_1,\ldots,\beta_{N},0\}$ obtained from $\alpha,\beta$ by adding $\alpha_{N}>0$ and $0$ at the end. By definition 
    \[
    \mathcal F_{N+1}(\alpha',\beta') = \mathcal F_N(\alpha,\beta) = T_1(N).
    \]
    At the same time, by Theorem~\ref{thm:conf_exist} one has 
    \[
    T_1(N)=\mathcal F_{N+1}(\alpha',\beta')<T_1(N+1,\alpha')\leq T_1(N+1)
    \]
    and the step of induction is proved.
\end{proof}

\subsection{Proof of Theorem~\ref{thm:rot_symm}}

Item (2) is proved in Section~\ref{sec:multiplicity}, item (3) follows from~\eqref{eq:St_trans_W}. Item (5) is proved in the previous section while item (6) follows from item (1) and Corollary~\ref{cor:FCMS_uniqueness}.

It remains to show (1) and (4). Once again, uniqueness follows from Corollary~\ref{cor:FCMS_uniqueness}. Let $\Phi_N\colon(\Sp,\Gamma^{(N)})\to\B^3$ be the map from that Corollary. Denoting $\mC(N)=\Phi_N(\Sp)$ one has
\begin{equation*}
    \begin{split}
2\area (\mC(N)) &= \int_{\Sp}|d\Phi_N|^2_g\,dv_g = \int_{\Gamma^{(N)}}((\partial_{n_+}+\partial_{n_-})\Phi_N,\Phi_N)\,ds_g \\
&=\bar\tau_1(M,\Gamma^{(N)},g_{\Phi_N}),
 \end{split}
\end{equation*}
where in the last equality we used the definition of $g_{\Phi_N}$ along with the fact that $\tau_1(\Sp,\Gamma^{(N)},g_{\Phi_N})=1$ and components of $\Phi_N$ are the corresponding eigenfunctions.

The particular construction of $\mC(2)$ and the expression for its area is given in Appendix~\ref{app}.

\begin{remark}
    Finally, we remark that the results of Section~\ref{sec:existence} can be extended to other surfaces admitting the action of $\mathbb{S}^1$, namely $\mathbb{RP}^2$, torus and Klein bottle. The latter two cases are even simpler compared to $\Sp$ due to the fact that the action is free. 
\end{remark}
\color{black}
\appendix

\section{Computational proof of Theorem~\ref{thm:rot_symm} for $N=2$} \label{app}
In this section, we present a nontrivial example of the computation of the Steklov transmission spectrum. Furthermore, we provide a direct proof of Theorem~\ref{thm:rot_symm} for $N=2$ inspired by the optimisation of the first Steklov eigenvalue on the rotationally symmetric annuli in~\cite{FS0}.

\subsection{Computation of eigenvalues}
 Consider the cylindrical domain 
\[
C = \mathbb{S}^1 \times (0, T),
\]
where \( T > 0 \) is a parameter. Define the curves:
\[
\Gamma_1 = \mathbb{S}^1 \times \{0\}, \quad \Gamma_2 = \mathbb{S}^1 \times \{T\}, \quad \Gamma = \Gamma_1 \cup \Gamma_2.
\]
Additionally, let
\[
D_1 = \mathbb{D} \times \{0\}, \quad D_2 = \mathbb{D} \times \{T\},
\]
and define the full domain as
\[
M_T = D_1 \cup C \cup D_2.
\]
See Figure \ref{Closed_cylinder}.

\begin{figure}[h]
\centering
\includegraphics[width=5cm]{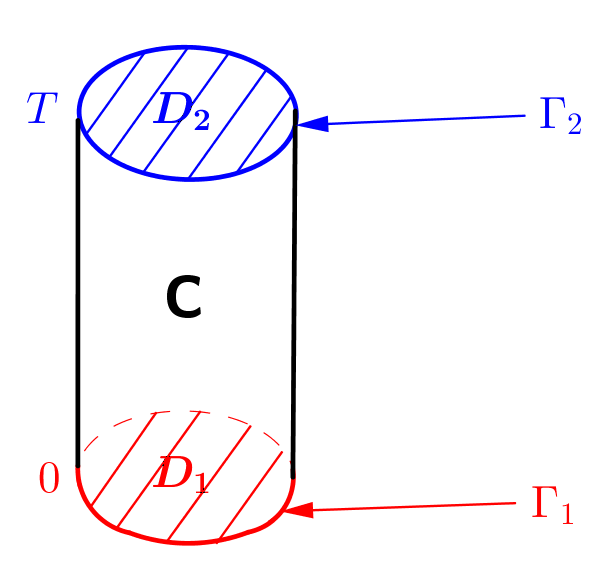}
\caption{Closed cylinder, \( N=2 \).}
\label{Closed_cylinder}
\end{figure}

On \( M_T \), the metric \( g \) is defined as follows.
\begin{itemize}
\item On the cylindrical region \( C \):
  \[
  g|_C = f^2(t) (dt^2 + d\theta^2), \quad t \in (0,T), \quad \theta \in \mathbb{S}^1,
  \]
  where \( f > 0 \) is smooth on \( \overline{C} \).

\item On the disk \( D_1 \):
  \[
  g|_{D_1} = f_1^2(x,y) (dx^2 + dy^2),
  \]
  where \( f_1 > 0 \) is smooth on \( \overline{D_1} \) and satisfies the condition
  \[
  f_1^2(x,y) = f(0), \quad (x,y) \in \Gamma_1.
  \]

\item On the disk \( D_2 \):
  \[
  g|_{D_2} = f_2^2(x,y) (dx^2 + dy^2),
  \]
  where \( f_2 > 0 \) is smooth on \( \overline{D_2} \) and satisfies the condition
  \[
  f_2^2(x,y) = f(T), \quad (x,y) \in \Gamma_2.
  \]
\end{itemize}

\begin{theorem}\label{thm_spectrum_N=2}
The Steklov transmission eigenvalues of $(M_T, g, \Gamma)$ are given by
\begin{align*}
    \tau_0 &= 0, \quad \tau_1 = \frac{f(0)+f(T)}{Tf(0)f(T)}, \\
    \tau_n^- &= \frac{ne^{nT}(f(0)+f(T))-\sqrt{\Delta_n}}{2f(0)f(T)\sinh(nT)}, \\
    \tau_n^+ &= \frac{ne^{nT}(f(0)+f(T))+\sqrt{\Delta_n}}{2f(0)f(T)\sinh(nT)}, 
\end{align*}

where $n>0$ and
\[
\Delta_n = n^2 e^{nT} \left[e^{nT}(f(0)+f(T))^2 - 8f(0)f(T)\sinh(nT) \right].
\] 

Each eigenvalue $\tau_n^+$ and $\tau_n^-$ is associated with an eigenspace generated by $u_n^\pm=H^{n,\pm} \beta^{n}(\theta), \theta\in\mathbb{S}^1$, where $\beta^{n}$ is a linear combination of $\sin(n\theta)$ and $\cos(n\theta)$. Furthermore, 
\[
H^{n,\pm} = 
\left\{ \begin{array}{lc}
h^{n,\pm}_0(0)r^n & \text{on $D_1$}\\
h^{n,\pm}_0(t) & \text{on $C$}\\
h^{n,\pm}_0(T) r^n & \text{on $D_2$},\\
\end{array}\right.
\]
where $h_0^{n,\pm}(t)$ satisfies the second-order ordinary differential equation
\[
h''(t) - n^2h(t) = 0,
\]
with boundary conditions
\[
\begin{cases}
\tau_n^{\pm} f(0) h(0) + h'(0) - n h(0) = 0, \\
\tau_n^{\pm} f(T) h(T) - h'(T) - n h(T) = 0.
\end{cases}
\]

Moreover, the first nonzero eigenvalue is given by:
\begin{align*}
    \tau_1(M_T, g, \Gamma) &= \min \Bigg\{ 
    \frac{f(0)+f(T)}{Tf(0)f(T)}, \\
    &\quad \frac{4}{f(0)+f(T) + 
    \sqrt{(f(0)+f(T))^2 - 8f(0)f(T)\sinh(T)e^{-T}}} 
    \Bigg\}.
\end{align*}

\end{theorem}

\begin{proof} Since the group $\mathbb{S}^1$ acts on $(M_T,\Gamma, g)$ by isometries and the Steklov transmission eigenvalues are eigenvalues of a pseudodifferential operator on $L^2(\Gamma)$, its eigenspaces are decomposed into the direct sums of representations of $\mathbb{S}^1$, i.e. 
it is sufficient to look for eigenfunctions of the following form
\[
u =
\begin{cases}
u_0(t, \theta) = h_0(t) \beta(\theta), & t \in [0,T], \quad \theta \in \mathbb{S}^1, \\
u_1(r\cos\theta, r\sin\theta, 0) = h_1(r) \beta(\theta), & r \in [0,1], \quad \theta \in \mathbb{S}^1, \\
u_2(r\cos\theta, r\sin\theta, T) = h_2(r) \beta(\theta), & r \in [0,1], \quad \theta \in \mathbb{S}^1,
\end{cases}
\]
where
\[
\beta(\theta) = b\cos(n\theta) + b'\sin(n\theta)
\]
and 
\[
h_i = 
\left\{ \begin{array}{lc}
a_i r^n & \text{on $D_i$}\\
a\cosh(nt) + a'\sinh(nt) & \text{on $C$}\\
\end{array}\right.
\]
for $n>0$, and $h_0(t) = a + a't$ if $n=0$.

Since the eigenfunctions are continuous across $\Gamma_i$ we have
\[
h_0(0) = h_1(1);\quad h(T) = h(1).
\]
Furthermore, the first order conditions on $\Gamma_i$ read
\[
\bd_rh(1) - \bd_th_0(0)  = \tau f(0)h_0(0)
\]
\[
\bd_rh_2(1) + \bd_th_0(T) = \tau f(T)h_0(T).
\]
Or, equivalently,
\[
\begin{cases}
 n h(0) - h'(0)  &= \tau f(0) h(0), \\
 n h(T) + h'(T)  &= \tau f(T) h(T).
\end{cases}
\]
It remains to see for which values of $\tau$ the function of the form $\alpha_0(t) = a\cosh(nt) + a'\sinh(nt)$ for $n>0$ or $\alpha_0(t) = a+a't$ can satisfy these boundary conditions.

Consider first the case 
 \( n = 0 \). Here, one has
\[
h_0(0) = a_0, \quad h_0'(0) = a_0', \quad 
h_0(T) = a_0 + a_0' T, \quad h_0'(T) = a_0'.
\]
Substituting into the system, we obtain the matrix equation:
\[
\begin{pmatrix}
\tau f(0) & 1 \\
\tau f(T) & \tau f(T)T - 1
\end{pmatrix}
\begin{pmatrix}
a_0 \\ a_0'
\end{pmatrix}
=
\begin{pmatrix}
0 \\ 0
\end{pmatrix}.
\]

This matrix equation has a nonzero solution if and only if 
\[
0 = \tau f(0)(\tau f(T)T - 1) - \tau f(T) = \tau [\tau T f(0) f(T) - (f(0) + f(T))],
\]
which simplifies to
\[
\tau = 0 \quad \text{or} \quad \tau = \frac{f(0) + f(T)}{T f(0) f(T)}.
\]
Thus, we arrive at two eigenvalues \( \tau^-_{0} = 0 \) and \( \tau^+_{0} = \frac{f(0) + f(T)}{T f(0) f(T)} \).

For \( n \geq 1 \), we have
\(
h_0(t) = a_0 \cosh(nt) + a'_0 \sinh(nt).
\)
Thus,
\[
h_0(0) = a_0, \quad h_0'(0) = n a'_0,
\]
\[
h_0(T) = a_0 \cosh(nT) + a'_0 \sinh(nT), \quad 
h_0'(T) = n (a_0 \sinh(nT) + a'_0 \cosh(nT)).
\]
This leads to the matrix equation:
\[
\begin{pmatrix}
\tau f(0) - n & n \\
\tau f(T) \cosh(nT) - n e^{nT} & \tau f(T) \sinh(nT) - n e^{nT}
\end{pmatrix}
\begin{pmatrix}
a_0 \\
a'_0
\end{pmatrix}
=
\begin{pmatrix}
0 \\
0
\end{pmatrix}.
\]
This matrix equation has a nonzero solution if and only if:
\[
0 = \tau^2 f(0) f(T) \sinh(nT) - n \tau e^{nT} (f(0) + f(T)) + 2 n^2 e^{nT}.
\]
The discriminant in \( \tau \) is:
\[
\Delta_n = n^2 e^{2nT} (f(0) + f(T))^2 - 8 n^2 e^{nT} f(0) f(T) \sinh(nT),
\]
\[
= n^2 e^{nT} \left[ e^{nT} (f(0) + f(T))^2 - 8 f(0) f(T) \sinh(nT) \right].
\]
Using the inequality \( \sinh t < e^t / 2 \), it follows that:
\[
\Delta_n > n^2 e^{nT} e^{nT} [(f(0) + f(T))^2 - 4 f(0) f(T)] = n^2 e^{nT} e^{nT} (f(0) - f(T))^2 \geq 0.
\]
Thus, there are always two different solutions given by:
\[
\tau^\pm_{n} = \frac{n e^{nT} (f(0) + f(T)) \pm \sqrt{\Delta_n}}{2 f(0) f(T) \sinh(nT)}.
\]
Clearly, $\tau^-_n<\tau_n^+$. To determine the smallest of the eigenvalues $\{\tau_n^-\}_{n>0}$, we observe that
\[
\tau^-_{n} = \frac{4n}{f(0) + f(T) + \sqrt{(f(0) + f(T))^2 - 8 f(0) f(T) \sinh(nT) e^{-nT}}}.
\]
Since $\sinh(nT)e^{-nT}$ is increasing with $n$, so does \( \tau^-_{n} \). Thus, the smallest non-zero eigenvalue $\tau_1$ is equal to $\min\{\tau_0^+,\tau_1^-\}$.

\end{proof}

In order to understand the behaviour of $\tau_1$, we compare \( \tau^+_{0} \) and \( \tau^-_{1} \) for different values of \( T \). We have

\[
\frac{\tau^-_{1}}{\tau^+_{0}} = \frac{T f(T) f(0)}{f(T) + f(0)} 
\cdot \frac{4}{f(0) + f(T) + \sqrt{(f(0) + f(T))^2 - 8 f(0) f(T) \sinh(T) e^{-T}}}
\]

\[
= \frac{4T}{\frac{1}{\alpha} + 1} 
\cdot \frac{1}{\alpha + 1 + \sqrt{(\alpha + 1)^2 - 8 \alpha \sinh(T) e^{-T}}} =: H(T),
\]
where \( \alpha = \frac{f(0)}{f(T)} \). If \( \alpha \) is fixed, then \( H(T) \) is increasing in \( T \) and satisfies
\[
\lim_{T \to 0^+} H(T) = 0, \quad \lim_{T \to +\infty} H(T) = +\infty.
\]
Thus, there exists a unique \( T(\alpha) \in \mathbb{R} \) such that $H(T(\alpha)) = 1$ and
\[
\tau_1=
\begin{cases}
\tau_1^-,&\text{ if } T\leq T(\alpha)\\
\tau_0^+,&\text{ if } T\geq T(\alpha).\\
\end{cases}
\]
In particular, for \( T > T(\alpha) \), the multiplicity of \( \tau_1 \) is 1;  
for \( T < T(\alpha) \), the multiplicity is 2;  
and for \( T = T(\alpha) \), the multiplicity is 3.  

Furthermore, for a fixed $\alpha$, then \( \tau^-_{1} \) is increasing in \( T \), whereas \( \tau^+_{0} \) is decreasing in \( T \). Thus, $\tau_1$ is maximized precisely when $T=T(\alpha)$, see Figure \ref{T_alpha}.

\begin{figure}[!h]
\centering
\includegraphics[width=12.75cm]{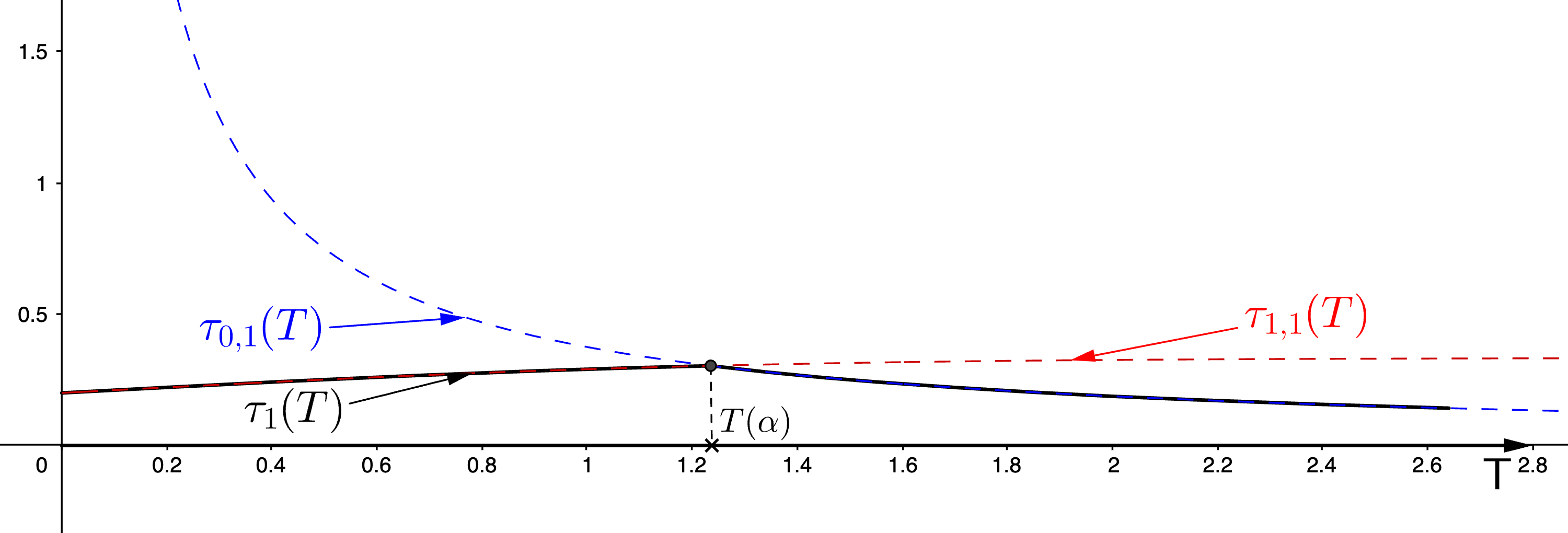}
\caption{Curve of \( \tau_1 \) for \( \alpha = 2 \).}
\label{T_alpha}
\end{figure}

\subsection{Eigenvalue optimization} As we have seen from the previous section the function $\bar\tau_1$ in the class of rotationally symmetric metric depends on two parameters $\alpha,T$. For a fixed $\alpha$ the function $\bar\tau_1(\alpha,T)$ is maximized for $T=T(\alpha)$, so to find the supremum of $\bar\tau$ it remains to study the function $\bar\tau(\alpha, T(\alpha))$. 

\begin{theorem}\label{thm_MT_bound}
For a given \( \alpha \in (0,+\infty) \), there exists a unique \( a \in \mathbb{R} \) such that 
\[
T(\alpha) = t_1(a) + t_2(a), \quad \text{and} \quad \alpha = \frac{t_2(a)}{t_1(a)},
\]
where \( t_1(a), t_2(a) \) are the positive solution to the following two equations respectively
  \[
  \label{eq:t1t2}
  t = \cosh(t + a) e^{-(t + a)};\qquad t = \cosh(t - a) e^{-(t - a)}.
  \]
Furthermore, for all $T>0$ one has
\[
\bar\tau_1(\alpha, T) \leq 2\pi \left[\frac{1}{t_1(a)} + \frac{1}{t_2(a)}\right].
\]
\end{theorem}
\begin{proof}

We have shown in the previous section that $T(\alpha)$ is characterized as the unique value of $T$ for which $\tau_1(\alpha,T)$ has multiplicity $3$. The theorem is proved by giving an explicit construction of such $T$.
The example is built from a piece of a catenoid $C_a$, which is obtained as a surface of revolution of the graph $y=h_a(t):=\cosh(t-a)$ parametrized by $t\in(-t_1(a),t_2(a))$ and the angle of revolution $\theta$.

\begin{lemma}
\label{lem:t1t2}
    One has that \( t_2(a) \) is such that the line connecting connecting $(t_2(a), h_a(t_2(a)))$ to the origin is the principal bisector of the angle $\angle A$ between the tangent line to the curve \( y = \cosh(t - a) \) at the point \( (t_2(a), h_a(t_2(a))) \), and the vertical line passing through \( (t_2(a), h_a(t_2(a))) \). Moreover, if $2\alpha_2(a)$ denotes the measure of $\angle A$, then
    \[
    \cos\beta_2(a) = \frac{1}{\sqrt{1 + e^{-2(t_2(a)-a)}}}
    \]
    The point $t_1(a)$ and the angle $\beta_1(a)$ are defined as $t_2(-a)$ and $\beta_2(-a)$ respectively, see Figure \ref{bisector}.
\end{lemma}
\begin{figure}
\centering
\includegraphics[width=12cm]{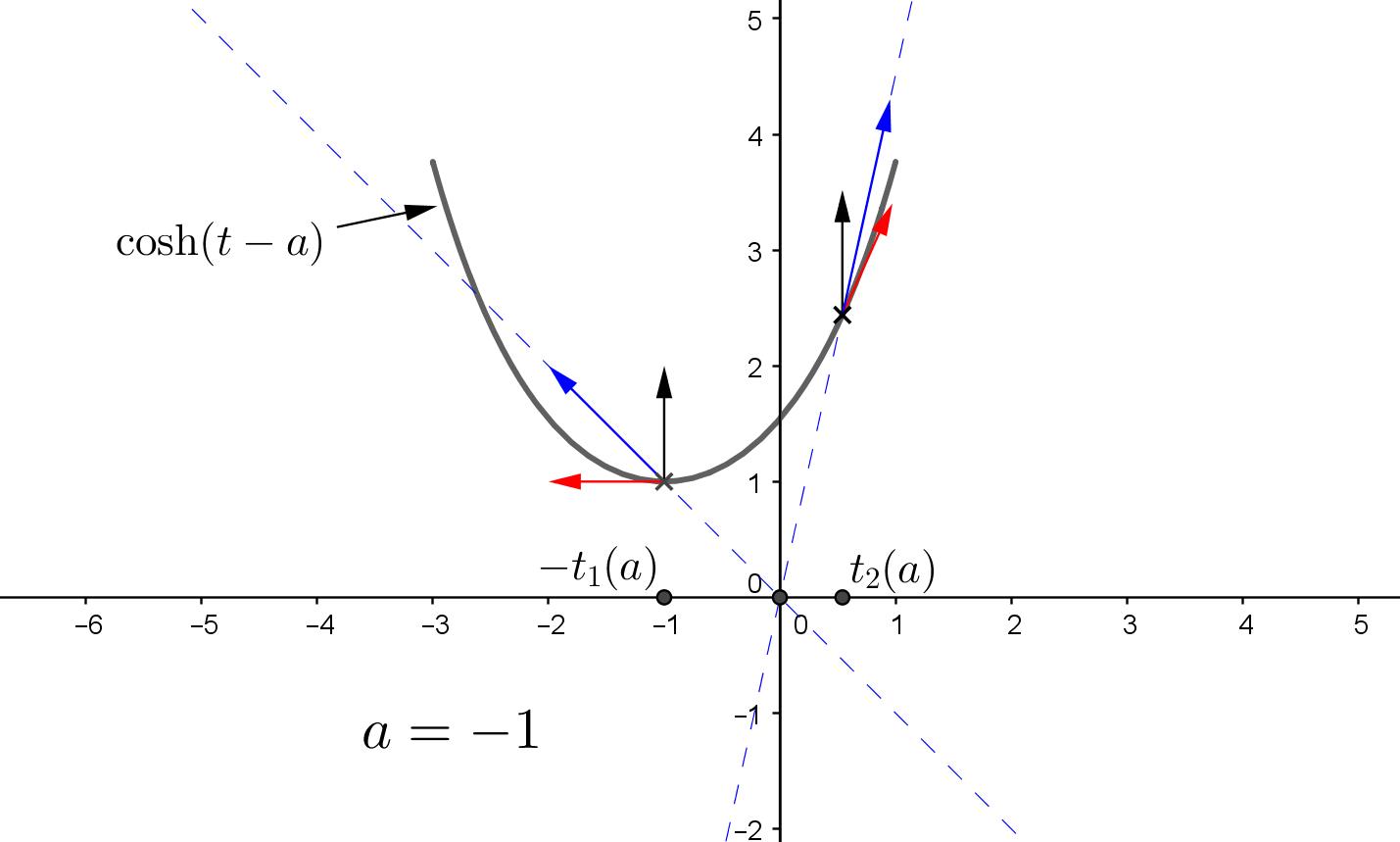}
\caption{Principal bisector}
\label{bisector}
\end{figure}
\begin{proof}
The tangent at \( (t_2(a), h_a(t_2(a))) \) is determined by the unit vector.
\[
v_1 = \frac{1}{\sqrt{1 + (h_a'(t))^2}}
\begin{pmatrix}
1 \\
h_a'(t)
\end{pmatrix}
= \frac{1}{\cosh(t - a)}
\begin{pmatrix}
1 \\
\sinh(t - a)
\end{pmatrix}.
\]
Thus, the parametric equation of the principal bisector of $\angle A$ is given by
\[
\begin{pmatrix}
t \\
h_a(t)
\end{pmatrix}
+ s 
\begin{pmatrix}
\frac{1}{\cosh(t - a)} \\
1 + \tanh(t - a)
\end{pmatrix} =
\begin{pmatrix}
t + \frac{s}{\cosh(t - a)} \\
\cosh (t + a) + s (1 + \tanh(t - a))
\end{pmatrix}.
\]
The bisector passes through \( (0,0) \) if and only if
\[
t \cosh (t - a) = -s = \frac{\cosh (t - a)}{1 + \tanh(t - a)},
\]
which is equivalent to
\[
t = \cosh(t - a) e^{-(t - a)}.
\]

To compute the expression for the angles we observe that by elementary trigonometry we have
\[
\tan(\beta_1) = \frac{-t}{\cosh(t - a)} \Big|_{t=-t_1(a)} = \frac{t_1(a)}{\cosh(t_1(a) + a)} = e^{-(t_1(a) + a)}.
\]
Thus,
\[
\cos(\beta_1) = \frac{1}{\sqrt{1 + \tan^2(\beta_1)}} = \frac{1}{\sqrt{1 + e^{-2(t_1(a) + a)}}}.
\]
Similarly, we obtain the expression for $\beta_2(a)$.
\end{proof}

Before continuing with the proof, let us establish the properties of $t_1,t_2$ as functions of $a$.
\begin{lemma} 
\label{lem:tia}
For all \( a \in \mathbb{R} \), the equation 
\[
t = \cosh(t + a) e^{-(t + a)}
\]
has a unique solution \( t_1(a) \), and the equation 
\[
t = \cosh(t - a) e^{-(t - a)}
\]
has a unique solution \( t_2(a) = t_1(-a) \). Furthermore, $t_1(a)$ is strictly decreasing while $t_2(a)$ is strictly increasing with
\[
\lim_{a \to -\infty} t_1(a) = +\infty, \quad \lim_{a \to +\infty} t_1(a) = \frac{1}{2},
\]
\[
\lim_{a \to -\infty} t_2(a) = \frac{1}{2}, \quad \lim_{a \to +\infty} t_2(a) = +\infty.
\]
\end{lemma}

\begin{proof}
Let \( a \in \mathbb{R} \) and define the function
\[
f(t) = t - \cosh(t + a) e^{-(t + a)} = t - \frac{1}{2} (1 + e^{-2(t + a)}),
\]
Computing the derivative, we obtain
\[
f'(t) = 1 + e^{-2(t + a)} > 0.
\]
Moreover, one has
\[
\lim_{t \to -\infty} f(t) = -\infty, \quad \lim_{t \to +\infty} f(t) = +\infty,
\]
Since both \( f(t) \) is strictly increasing and pass from negative to positive values, the equation \( f(t) = 0 \) has a unique solution  \( t_1(a) \). Since the equation $t = \cosh(t + a) e^{-(t + a)}$ is transformed into $t = \cosh(t - a) e^{-(t - a)}$ after changing $a\mapsto (-a)$, we have that $t_2(a)$ is uniquely defined and $t_2(a) = t_1(-a)$.

The function $t_1(a)$ is differentiable by the implicit function theorem and
\[
t_1' = -(t_1' + 1) e^{-2(t_1 + a)} = -(t_1' + 1) (2t_1 - 1).
\]
Rearranging, we obtain
\begin{equation}
\label{eq:t1'}
t_1' = -1 + \frac{1}{2t_1}.
\end{equation}
Since 
\[
t_1 = \frac{1}{2} (1 + e^{-2(t_1 + a)}) > \frac{1}{2},
\]
it follows that \( t_1' < 0 \). In particular, the limits of $t_1(a)$ as $a\to\pm\infty$ exist. Using $t_1(a) = \frac{1}{2}(1+e^{-2(t_1(a) + a)})$ and $t_1(a)>\frac{1}{2}$ they can be easily computed to be equal to $+\infty$ and $\frac{1}{2}$ respectively. The limits for $t_2(a)$ follow from $t_2(a) = t_1(-a)$.
\end{proof}

Let $M_a$ be the surface obtained by attaching two round caps to $C_a$ at $x=-t_1(a)$ and $x=t_2(a)$ along $\Gamma_1$ and $\Gamma_2$ endowed with the induced metric $g_a$. This is a rotationally symmetric surface with $\Gamma = \Gamma_1\cup\Gamma_2$, which is minimal on $M_a\setminus\Gamma$. Furthermore, according to Lemma~\ref{lem:t1t2}, the position vector $\vec{X}$ on $M_a$ satisfies $(\partial_{n_1}+ \partial_{n_2}) \vec{X}\parallel \vec{X}$ along $\Gamma$ and hence
\[
\partial_{n_1} X + \partial_{n_2} X = l_i X \quad \text{on } \Gamma_i,
\]
where
\[
l_i = \frac{2 \cos(\beta_i)}{|X|} =\frac{2 \cos(\beta_i)}{\sqrt{t_i^2+ \cosh((-1)^{i}t_i-a)^2}}.
\]
Computing using the formulas for $\beta_i$ yields
\[
l_2 = \frac{1}{t_2\cosh(t_2-a)},\quad l_1 = \frac{1}{t_1\cosh(t_1+a)}.
\]

Let $\rho$ be any function such that $\rho = l_i$ on $\Gamma_i$. Then $\vec X$ forms a $3$-dimensional eigenspace for the Steklov transmission eigenvalue $\tau = 1$ on $(M_a,\Gamma,\rho g_a)$. Obviously, $T(a)=t_1(a) + t_2(a)$, so it remains to show that $\tau_1 = 1$, compute $\alpha$ and $\bar\tau_1(M_a,\Gamma,\rho g_a)$.

To start, note that 
\[
\ell_{\rho g_a}(\Gamma_i) = l_i \ell_{g_a}(\Gamma_i) = 2\pi l_i\cosh((-1)^{i}t_i-a) = \frac{2\pi}{t_i},
\]
so that $\alpha(a) = t_2(a)/t_1(a)$.
Setting $f(0) = \frac{1}{t_1}$ and $f(T) = \frac{1}{t_2}$ in the formula for $\tau_0^+$ and $\tau_1^-$ we obtain using $T = t_1+t_2$ that
\[
\tau_0^+ = \tau_1^- =  1,
\]
where in the latter computation we used $e^{-2t_1(a)} = e^{2a}(2t_1(a)-1),$ $e^{-2t_2(a)} = e^{-2a}(2t_2(a)-1)$. This implies $\tau_1(M_a,\Gamma,\rho g_a) =\tau = 1$ and it corresponds to the $3$-dimensional eigenspace, so that $T(\alpha(a)) = t_1(a) + t_2(a)$ and, finally,
\[
\bar\tau_1(\alpha(a), T(\alpha(a))) = 2\pi\left(\frac{1}{t_1(a)} + \frac{1}{t_2(a)}\right).
\]

Finally, we observe that by the properties of $t_1(a)$ and $t_2(a)$ proved in Lemma~\ref{lem:tia} the function $\alpha(a) = \frac{t_2(a)}{t_1(a)}$ is strictly increasing with 
\[
\lim_{t\to-\infty} \alpha(a) = 0;\qquad \lim_{t\to+\infty} \alpha(a) = +\infty.
\]
As a result, for any $\alpha_0>0$ there exists a unique $a$ such that $\alpha(a)=\alpha_0$, which completes the proof of the theorem.

\end{proof}

Theorem~\ref{thm_MT_bound} reduces Theorem~\ref{thm:rot_symm} for $N=2$ to the question of finding the maximum of the following function
\[
F(a):= \frac{1}{2\pi}\bar\tau_1(\alpha(a), T(\alpha(a))) = \left(\frac{1}{t_1(a)} + \frac{1}{t_2(a)}\right)
\]
as a function of $a\in\R$.

\begin{theorem}\label{thm_max_N=2}
The function $F(a)$ achieves its maximum at $a=0$. As a result, one has  
\[
\bar\tau_1(\alpha, T)\leq \frac{4\pi}{t_1(0)},
\]
where $t_1(0)\approx 0.639$ is the unique positive solution of \( t = \cosh(t) e^{-t}\). The equality is achieved iff $\alpha = 0$ and $T=2t_1(0)$.
\end{theorem}



\begin{proof}
Observe that $F(a)$ is an even function. Therefore, it suffices to show that \( F(a) \leq F(0) \) for all \( a \geq 0 \). By~\eqref{eq:t1'} we have
\begin{equation*}
\begin{split}
F'(a) &=  -t_1(a)^{-2}(-1 + (2t_1(a))^{-1}) - t_2(a)^{-2}(1 - (2t_2(a))^{-1})\\
&= 4 \left[ Q\left((2t_1(a))^{-1}\right) - Q\left((2t_2(a))^{-1}\right) \right],
\end{split}
\end{equation*}
where \( Q(x): = x^2 (1 - x) \). Since \( (2t_1(a))^{-1}, (2t_2(a))^{-1} \in (0,1) \) by Lemma~\ref{lem:tia}, we study \( Q \) on \( (0,1) \).

The derivative of \( Q \) is given by:
\[
Q'(x) = 2x(1-x) - x^2 = x(2 - 2x - x) = x(2 - 3x).
\]
Thus, \( Q(x) \) is increasing on \( (0, 2/3) \) and decreasing on \( (2/3, 1) \). Since \( t_1(0) \approx 0.639 < 0.7 \), it follows that \( (2t_1(0))^{-1} > \frac{5}{7} > \frac{2}{3} \).

Since $Q(0) = Q(1) = 1$ there exists a unique \( 0 < a_0 < \frac{2}{3} \) be such that
\[
Q((2t_2(a_0))^{-1}) = Q((2t_1(0))^{-1}).
\]
For \( 0 < a < a_0 \), we have:
\[
t_1(a) < t_1(0) \Longrightarrow (2t_1(a))^{-1} > (2t_1(0))^{-1} > \frac{2}{3}
\]
And, thus, 
\[
 Q((2t_1(a))^{-1}) < Q((2t_1(0))^{-1}).
\]
On the other hand,
\[
t_2(0) < t_2(a) < t_2(a_0) \Longrightarrow (2t_2(0))^{-1} > (2t_2(a))^{-1} > (2t_2(a_0))^{-1}.
\]
If \( (2t_2(a))^{-1} < \frac{2}{3} \), then \( Q((2t_2(a))^{-1}) > Q((2t_2(a_0))^{-1}) = Q((2t_2(0))^{-1}) \). Otherwise, if \( (2t_2(a))^{-1} \geq \frac{2}{3} \), then \( Q((2t_2(a))^{-1}) > Q((2t_2(0))^{-1}) \). This implies
\[
Q((2t_1(a))^{-1}) - Q((2t_2(a))^{-1}) < Q((2t_1(0))^{-1}) - Q((2t_2(0))^{-1}) = 0.
\]
In either case,  \( F \) is decreasing on \( (0, a_0) \), and we obtain \( F(a) < F(0) \) for all \( 0 < a < a_0 \). 

For \( a \geq a_0 \), we estimate:
\[
F(a) = \frac{1}{t_1(a)} + \frac{1}{t_2(a)} < 2 + \frac{1}{t_2(a_0)}.
\]
To compute \( t_2(a_0) \), let \( x_0 = (2t_2(0))^{-1} \). Then, \( (2t_2(a_0))^{-1} \) is a root of the cubic polynomial
\[
Q(x) - Q(x_0) = -x^3 + x^2 - Q(x_0).
\]
Using polynomial division, we obtain
\[
Q(x) - Q(x_0) = (x - x_0)(-x^2 + (1 - x_0)x + x_0(1 - x_0)).
\]
Solving for \( x \), we find that the other positive root is given by
\[
x_2 = \frac{1}{2} \left( (1 - x_0) + \sqrt{(1 - x_0)(1 + 3x_0)} \right).
\]
Thus, \( (2t_2(a_0))^{-1} = x_2 \), and we obtain
\[
2 + \frac{1}{t_2(a_0)} = 2 + 2x_2 \approx 2 + 1.07085 = 3.07085.
\]
Since \(2t_1(0)^{-1}\approx 3.12989,\)
it follows that
\[
2 + \frac{1}{t_2(a_0)} < \frac{2}{t_1(0)},
\]
which completes the proof.
\end{proof}

\section{Rotationally symmetric minimal surfaces in $\B^3$}
\label{app2}
\subsection{Minimal Surfaces of Revolution in \(\mathbb{R}^3\)}

In this section we study properties of a $\bar\tau_1$-critical map by the first $\tau_1$-eigenfunctions
\[
\Phi_N(z,\theta) = (a_N(z)\cos(\theta),a_N(z)\sin\theta,b_N(z)),
\]
where $(z,\theta)\in\R\times\mathbb{S}^1$, $\Gamma_i = \{z = z_i\}$, $i=1,\dots,N$ and we set $z_0 = -\infty$, $z_{N+1} = +\infty$ for convenience. Moreover, $b_N(z)$ is a piecewise linear function with break points at $z_i$.
Our goal is to provide missing details from the proof of Corollary~\ref{cor:FCMS_uniqueness}, namely to show that $\Phi_N$ is injective and that its image is uniquely determined up to an element of $O(3)$. The uniqueness statement relies heavily on the work of~\cite{KZ}.

Let us make some preliminary observations. First, $\Phi_N|_{\{z_i<z<z_{i+1}\}}$ is a smooth conformal map, so that its injective there and the image is a minimal surface of revolution about the $z$-axis. It is well known that such minimal surface has to be a piece of a horizontal plane (a disk or an annulus) or a piece of catenoid obtained by rotating a catenary $(t, r_{a,b}(t))$, where
\[
r_{a,b}(t) = a \cosh\left(\frac{t - b}{a}\right), \quad t\in[t_1,t_2]
\]
and $a>0$, $b\in\R$. Second, since $\Phi_N$ is conformal and $r_{a,b}(t)\to+\infty$ as $t\to\pm\infty$, the images $\Phi_N(\{z>z_N\})$ and $\Phi_N(\{z<z_1\})$ have to be horizontal discs. Hence, $b_N(z)$ is locally constant for $z>z_N$ and $z<z_1$. At the same time, if $i\ne 0,N$, then $\Phi_N(\{z_i<z<z_{i+1}\})$ can not lie in a horizontal plane as horizontal plane intersects the sphere $\Sp$ in a single circle whereas $\Phi_N(\Gamma_i)$ and $\Phi_N(\Gamma_{i+1})$ both have to lie on $\Sp$. Thus, we conclude that the the image of $\Phi_N$ looks like a collection of $(N-1)$ catenoids with two planar caps, which intersect at the common boundary on $\Sp$ at equal and opposite angles. Finally, we observe
\begin{proposition}
\label{prop:b_N_monotone}
    The function $b_N(z)$ is strictly monotone on $[z_1,z_N]$. In particular, $\Phi_N$ is injective.
\end{proposition}
\begin{proof}
For each \(i\in\{1,\dots,N-1\}\) on \((z_i,z_{i+1})\) one has
\[
  b_N(z)=k_i\,(z-z_i)+m_i,
  \qquad z\in(z_i,z_{i+1}),
\]
for suitable constants \(k_i,m_i\).  The transmission conditions
$$\left\{\begin{array}{l}
-b_N'(z_1^+)\,=\,\tau_1\,b_N(z_1),\\

b_N'(z_i^-)-b_N'(z_i^+)\,=\,\tau_1\,b_N(z_i),
\quad i=2,\dots,N-1,\\

b_N'(z_N^-)\,=\,\tau_1\,b_N(z_N),
\end{array}\right.$$ 
then become
$$\left\{\begin{array}{l}
  -k_1\,=\,\tau_1\,b_N(z_1),\\
  k_{i-1}-k_i\,=\,\tau_1\,b_N(z_i),
    \quad i=2,\dots,N-1,\\
  k_{N-1}\,=\,\tau_1\,b_N(z_N).
\end{array}\right.$$

By the Courant nodal domain theorem, $b_N(z)$ has a unique zero $z^*$ and since $b_N$ is constant for $z>z_N$ and $z<z_1$, one has \(z^*\in(z_1,z_N)\) and, therefore, without loss of generality we may assume that 
\[
  b_N(z^*)=0,\quad b_N(z)<0\text{ for }z<z^*,\quad b_N(z)>0\text{ for }z>z^*.
\]
In particular,
\[
  k_1=-\tau_1\,b_N(z_1)>0,
  \quad
  k_{N-1}=\tau_1\,b_N(z_N)>0.
\]
Let \(s\in\{2,\dots,N\}\) be the smallest index with \(z_s>z^*\).  Then for \(i\le s-1\) we have \(b_N(z_i)\le0\), so
\[
  k_{i-1}-k_i=\tau_1\,b_N(z_i)\le0
  \quad \text{i.e} \quad
  k_{i-1}\le k_i.
\]
Since \(k_1>0\), it follows that
\[
  0<k_1\le k_2\le\cdots\le k_{s-1}.
\]
Similarly, for \(i\ge s\), \(b_N(z_i)\ge0\) implies \(k_{i-1}\ge k_i\), and hence
\[
  k_{s-1}\ge k_s\ge\cdots\ge k_{N-1}>0.
\]
Thus \(k_i>0\) for all \(i\), and consequently \(b_N\) is strictly increasing on \([z_1,z_N]\).

\end{proof}

\subsection{Stacks of catenoids}
The image $\Phi_N(\Sp)$ contains a collection of catenoids each of which intersects the sphere at two circles. Since we are concerned with the angles of these intersections it is convenient to introduce angular characterization of catenaries as in~\cite{KZ}.

Let $\mathbb{K}_{a,b}\subset\R^2$ denote the catenary $\{(t,r_{a,b}(t))\}$ that intersects the circle $\mathbb{S}^1$ in two points $P_1,P_2$ with $P_1$ to the right of $P_2$. We define  $\beta^+(a,b), \alpha^+(a,b)$ as follows
\begin{enumerate}
  \item $\beta^+(a,b)$ is the angle between the positive $x$–axis and the radius vector $\overrightarrow{OP_1}$, i.e.\ the latitude of the intersection point $P_1\in\mathbb{K}_{a,b}\cap\mathbb{S}^1$.
  \item $\alpha^+(a,b)$ is the angle between the inward‐pointing tangent to the catenary at $P_1$, given by the vector $(-1,-r'_{a,b}(t_1))$, and the unit tangent $\overrightarrow{OP_1}^\perp$ to $\mathbb{S}^1$ at $P_1$, oriented towards $(-1,0)$. The orientation is fixed by requiring
  \[
    \overrightarrow{OP_1}\cdot\overrightarrow{OP_1}^\perp = 0,
    \qquad
    \det\bigl(\overrightarrow{OP_1},\,\overrightarrow{OP_1}^\perp\bigr)>0.
  \]
\end{enumerate}

\begin{lemma}[{\cite[Lemma 2.6]{KZ}}]
\label{lem:KZ}
    One has
    \[
  \beta^+(a,b)\in(0,\pi),
  \quad
  \alpha^+(a,b)\in(0,\pi-\beta^+(a,b)),
\]
Furthermore, for every pair $(\beta,\alpha)$ in this range there is a unique catenary $\mathbb{K}_{a,b}$ such that
\[
  \beta^+(a,b)=\beta,
  \quad
  \alpha^+(a,b)=\alpha.
\]
\end{lemma}

We denote the resulting catenary by $\mathbb{K}(\beta,\alpha)$. Similarly, $\beta^-(\beta,\alpha), \alpha^-(\beta,\alpha)$, which can be also considered as functions of $(a,b)$, are defined as follows:
\begin{enumerate}
  \item $\beta^-$ is the angle between the positive $x$–axis and the radius vector $\overrightarrow{OP_2}$, i.e.\ the latitude of the left intersection point $P_2\in\mathbb{K}(\beta,\alpha)\cap\mathbb{S}^1$.
  \item $\alpha^-$ is the angle between the inward‐pointing tangent $(1,r'_{a,b}(t_2))$ to the catenary at $P_2$ and the vector $-\overrightarrow{OP_2}^\perp$, where $\overrightarrow{OP_2}^\perp$ is the unit tangent to $\mathbb{S}^1$ at $P_2$, oriented towards $(1,0)$.
\end{enumerate}
These angles are also depicted in Figure~\ref{fig:angular_cat}.

\begin{figure}[h]
  \centering
  \includegraphics[width=12cm]{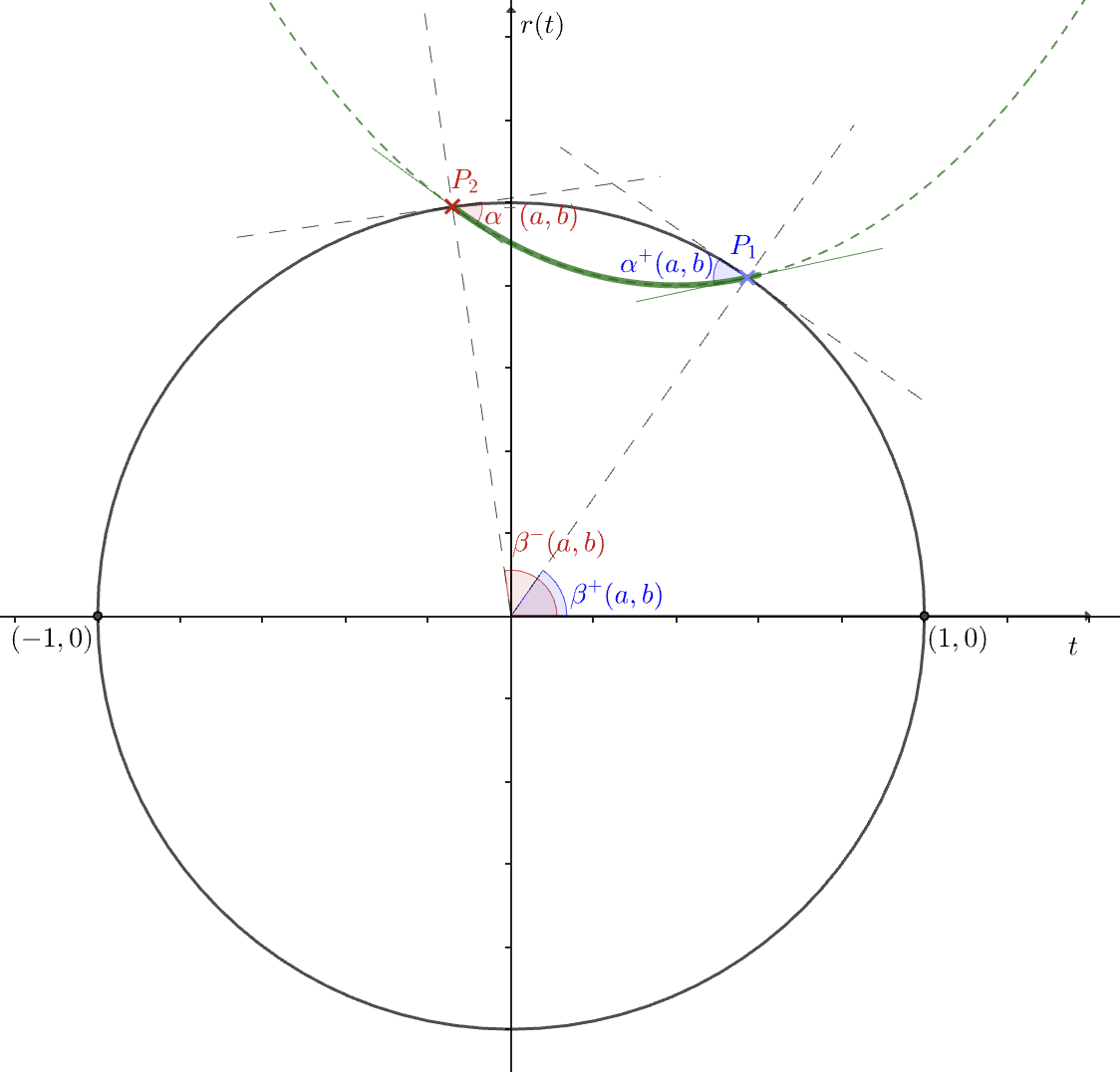}
  \caption{Angular characterization of a catenary with parameters $a=0.8$, $b=0.4$.}
  \label{fig:angular_cat}
\end{figure}

\subsection{Balanced configurations and uniqueness}

We now describe the properties of the image $\Phi_N(\Sp)$ in the notation introduced above. Namely,  $\Phi_N(\Sp)$ is a surface of revolution of the union of catenaries and lines in $\B^2$ with the following properties. Let $\beta_i\in\left[0,\pi\right]$ such that $(\cos\beta_i,\sin\beta_i )= (a_N(z_i),b_N(z_i))$ and without loss of generality assume $\beta_1<\pi/2$. Then by Proposition~\ref{prop:b_N_monotone} one has
\[
0<\beta_1<\beta_2<\cdots<\beta_N<\pi.
\]
In particular, for $i=1,\ldots, N-1$ the image $\Phi_N(\{z_i<z<z_{i+1}\})$ is a surface of revolution of $\mathcal C_i := \mathbb{K}(\beta_i,\alpha^+_i)\cap\B^2$, where $\alpha^+_i$ is such that  $\beta^-(\beta_i,\alpha^+_i)=\beta_{i+1}$. The fact that two adjacent catenoids meet the boundary $\bd\B^3$ at equal angles means that 
 for each  $i=2,\ldots,  N-1$,  $\alpha_{i}^+=\alpha^-(\beta_i,\alpha^+_i)$. Finally, the images $\Phi_N(\{z_N<z\})$ and $\Phi_N(\{z<z_1\})$ are disks and a surface of revolutions of line segments. This translates into $\alpha_1^+=\beta_1^+$ and $\alpha^-(\beta_{N-1},\alpha_{N-1}^+) = \pi-\beta_N$, see Figure~\ref{fig:cat_conf_k=3}. Based on this, we give the following definition.

\begin{figure}[ht]
  \centering
  \includegraphics[width=12cm]{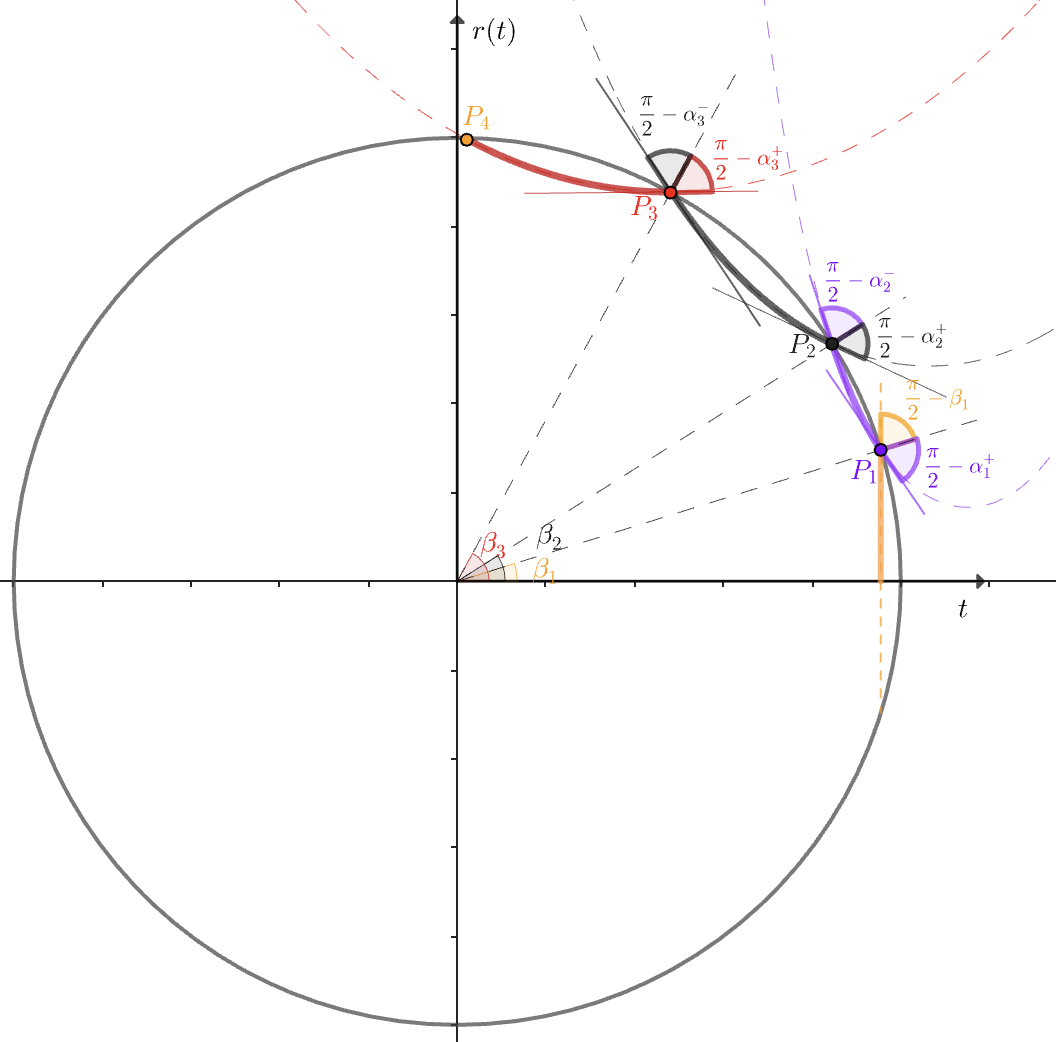}
  \caption{A collection of catenaries satisfying the angle conditions.}
  \label{fig:cat_conf_k=3}
\end{figure}

\begin{definition}\label{def:total-balanced}
Let $N\in\mathbb{N}$ and let
\[
0<\beta_1<\beta_2<\cdots<\beta_N<\pi
\]
be prescribed latitudes on $\mathbb{S}^1$. A \emph{balanced configuration} of order $N$ is a union
\[
\mathcal{C} \;=\;\bigcup_{i=0}^{N}\mathcal{C}_i
\]
with the following properties:
\begin{enumerate}
  \item $\mathcal{C}\cap\mathbb{S}^1=\{P_1,\dots,P_N\}$, where each $P_i$ has latitude $\beta_i$.
  \item $\mathcal{C}_0$ and $\mathcal{C}_N$ are the vertical segments joining $P_1$ and $P_N$ to the horizontal axis.
  \item For $1\le i\le N-1$, 
    \[
      \mathcal{C}_i \;=\;\mathbb{K}_{a_i,b_i}\,\cap\,\mathbb{B}^2
      \;=\;\mathbb{K}(\beta_i,\alpha_i^+)\,\cap\,\mathbb{B}^2,
    \]
    where $\alpha_i^+:=\alpha^+(a_i,b_i)\in(0,\pi-\beta_i)$ is such that $\beta^-(a_i,b_i)=\beta_{i+1}$
  \item At each interface $P_{i}$, the incoming and outgoing angles coincide
\[
\alpha_{i}^- \;=\;\alpha_{i}^+,
\quad i=1,\dots,N,
\] where  \[
\alpha_1^-:=\beta_1,
\qquad
\alpha_{i}^-:=\alpha^-(a_{i-1},b_{i-1}),
\qquad
\alpha_N^+:=\pi-\beta_N.
\]
\end{enumerate}
\end{definition}
An example of such a configuration can be found in Figures~\ref{Catenarydal} and~\ref{fig:sym_cat_conf_odd}.

\begin{figure}[h]
    \centering
    \includegraphics[width=12cm]{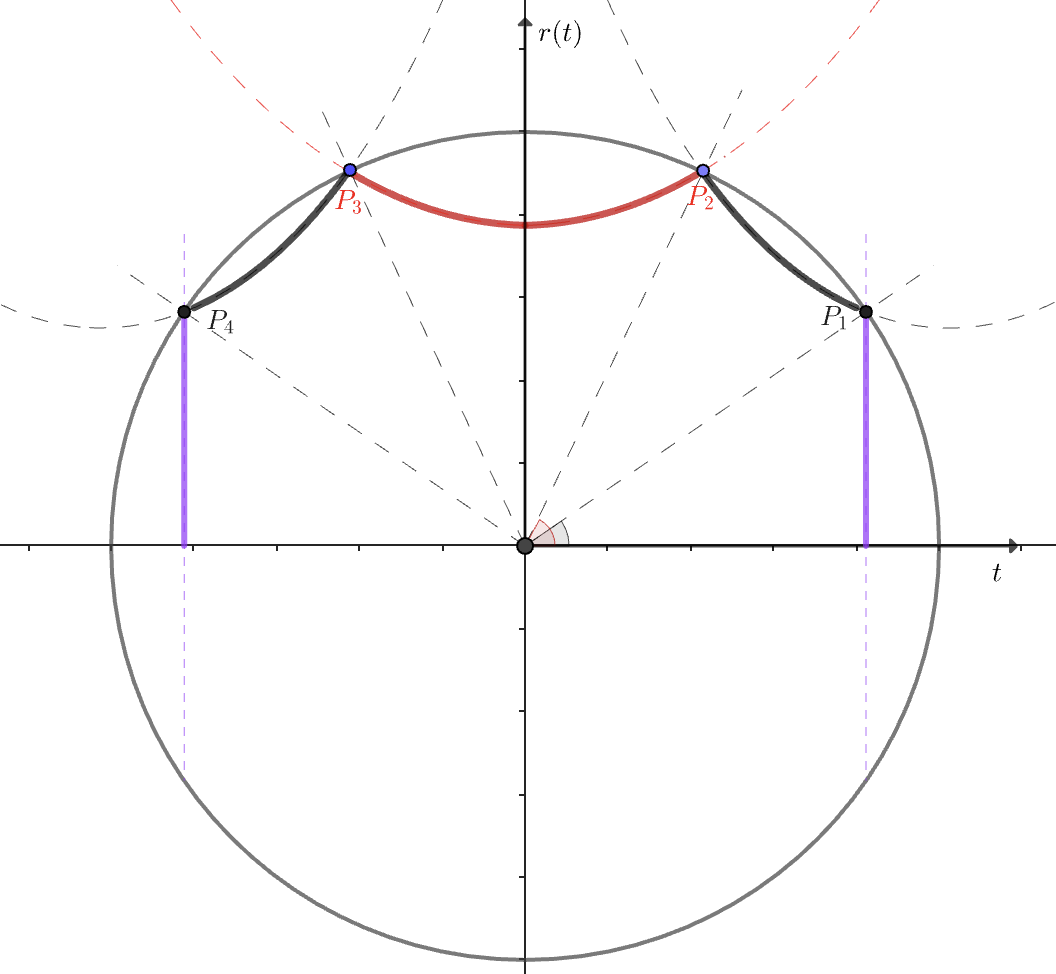}
    \caption{Symmetric balanced catenoidal configuration, \( k=2 \), \( N=2 \cdot 2 \)}
    \label{fig:sym_cat_conf_odd}
\end{figure}

The definition suggests the following procedure for constructing the balanced configurations.
Given an initial angle $\beta\in(0,\frac{\pi}{2})$, we construct sequences 
\[
  \bigl\{\beta_i(\beta)\bigr\}_{i\ge1},
  \quad
  \bigl\{\alpha_i(\beta)\bigr\}_{i\ge1}
\]
inductively as follows.
\begin{enumerate}
  \item For $i=1$, set
    \[
      \beta_1(\beta)=\beta,
      \quad
      \alpha_1(\beta)=\alpha_1^+=\beta.
    \]
    These determine a catenary $\mathbb{K}(\beta_1,\alpha_1)$ since $0<\beta<\pi-\beta$.  We then define
    \[
      \beta_2(\beta)=\beta^-(a_1,b_1),
      \quad
      \alpha_2(\beta)=\alpha^-(a_1,b_1)=\alpha_2^-.
    \]
  \item For $i\geq 2$, we construct inductively
    \begin{enumerate}
      \item If $\alpha_i(\beta)+\beta_i(\beta)>\pi$, by Lemma~\ref{lem:KZ} the catenary $\mathbb{K}(\beta_i,\alpha_i)$ does not exist and so the process terminates at this step and we obtained sequences
      \[
        \bigl\{\beta_j(\beta)\bigr\}_{j=1}^{i-1},
  \quad
  \bigl\{\alpha_j(\beta)\bigr\}_{j=1}^{i-1}
       \] 
      \item If $\alpha_i(\beta)+\beta_i(\beta)=\pi$, the catenary $\mathbb{K}(\beta_i,\alpha_i)$ does not exist, but addition of a vertical line completes the balanced configuration of order $i$.
      \item If $\alpha_i(\beta)+\beta_i(\beta)<\pi$, there exists a catenary $\mathbb{K}(\beta_i,\alpha_i)= \mathbb{K}_{a_i,b_i}$ and we set
        \[
          \beta_{i+1}(\beta)=\beta^-(a_i,b_i),
          \quad
          \alpha_{i+1}(\beta)=\alpha^-(a_i,b_i)=\alpha_{i+1}^-.
        \]
    \end{enumerate}
\end{enumerate}
This process yield a (possibly infinite) sequence of order $N(\beta)$ such that $\alpha_i(\beta)+\beta_i(\beta)\leq \pi$ with equality only if $i=N(\beta)$.
This (and much more general) type of sequences were studied in~\cite{KZ} and we summarize their results below.
\begin{proposition}
One always has $\alpha_i(\beta)<\tfrac\pi2$ for all $i$. The functions $\alpha_i(\beta)$, $\beta_i(\beta)$ are strictly increasing in $\beta$, whereas $N(\beta)$ is non-decreasing. Furthermore, $N(\beta)\to\infty$ as $\beta\to 0$, so that $\alpha_i(\beta)$, $\beta_i(\beta)$ are defined for sufficiently small $\beta$ and  $\alpha_i(\beta)$, $\beta_i(\beta)$ go to $0$ as $\beta\to 0$.
\end{proposition}
\begin{proof}
Let us first prove that $\alpha_i(\beta)<\tfrac\pi2$ for all $i$. Indeed, if 
$\beta_i(\beta)>\tfrac\pi2$ then $\alpha_i(\beta)<\pi-\beta_i<\tfrac\pi2$; 
If $\beta_i(\beta)\leq\tfrac\pi2$ then there are two possibilities. 
If $i=1$, then $\alpha_i(\beta)=\beta_i(\beta)<\tfrac\pi2$. If $i>1$ and $\beta_i(\beta)\leq\tfrac\pi2$, then it is shown in~\cite{KZ} that 
$\beta<\tfrac{2\pi}7$, in which case \cite[Lemma~2.19(c)]{KZ} implies 
$\alpha_i(\beta)\le\beta<\tfrac\pi2$.

Once $\alpha_i(\beta)<\tfrac\pi2$ is established, monotonicity of all the functions follows from \cite[Proposition 2.10]{KZ}. Limiting properties follow from \cite[Proposition 2.20]{KZ}.
\end{proof}

\begin{corollary}
\label{cor:uniq_B}
    For each $N\geq 2$ there is at most one balanced configuration of order $N$.
\end{corollary}
\begin{proof}
    Suppose there are two such configuration. By definition, all balanced configuration are generated by angles $\beta',\beta''\in \left(0,\frac\pi2\right)$. In particular, $N(\beta')= N(\beta'')$ and $\alpha_N(\beta') + \beta_N(\beta') = \alpha_N(\beta'') + \beta_N(\beta'') = \pi$. At the same time, assuming $\beta'<\beta''$, by monotonicity of $N$, the functions $\alpha_N(\beta),\beta_N(\beta)$ are defined on $[\beta',\beta'']$ and are strictly increasing there, which contradicts the fact that $\alpha_N(\beta) + \beta_N(\beta) = \pi$ has two distinct solutions.
\end{proof}

\subsection{Existence of balanced configurations}
The existence of a balanced configuration is proved in~\cite{KZ}, below we provide some details.

\begin{definition}
A {\em symmetric balanced configuration} is a balanced configuration \( \mathcal{C} \) such that for some $k$ either
\begin{itemize}
    \item \( \beta_{k+1} = \frac{\pi}{2}, \) or
    \item \( \beta_k + \beta_{k+1} = \pi. \)
\end{itemize}
In this case one has
\(\mathcal{C} =  \underline{R}\mathcal{C},\)
where \( \underline{R} \) denotes reflection with respect to the \( y \)-axis, is a . The total order of \( \mathcal{C} \) is then \( N = 2k + 1 \) in the first case, and \( N = 2k \) in the second case. See Figures~\ref{Catenarydal} and~\ref{fig:sym_cat_conf_odd}.
\end{definition}

In particular, it means that if $\beta$ is such that $\beta_k(\beta) = \pi/2$ or $\beta_k(\beta)+\beta_{k+1}(\beta) = \pi$, then $\beta$ corresponds to a symmetric balanced configuration of order $N = 2k+1$ or $2k$ respectively.

\begin{theorem}[{\cite[Corollary 2.24]{KZ}}]
\label{thm B}
For all \( N \geq 2 \), there exists a unique initial latitude \( \beta = \beta_1(N) \in \left( 0, \frac{\pi}{2} \right) \) such that the corresponding balanced configuration \( \mathcal{C}(\beta) \) is symmetric and has total order \( N \).
\end{theorem}

\begin{proof}[Sketch of Proof]
This result is fully proven in \cite{KZ}. We provide a sketch to illustrate the main ideas. The uniqueness is covered in Corollary~\ref{cor:uniq_B}. For the existence, the authors proceed by induction.

For \( N = 2 \), we constructed a symmetric balanced configuration corresponding to the critical drum in Appendix~\ref{app}.

Now suppose the result holds for some odd \( N = 2k-1 \), i.e., \( \beta_k(\beta) = \frac{\pi}{2} \). Then we have \( \beta_k(\beta) + \beta_{k+1}(\beta) > \pi \). By gradually lowering \( \beta \), we also decrease \( \beta_k(\beta) + \beta_{k+1}(\beta)\). Since $\beta_{k+1}(\beta)\to 0$, for sufficiently small $\beta$ one has $\beta_{k+1}(\beta)<\pi/2$. Therefore, for some intermediate $\beta$ one has \( \beta_k(\beta) + \beta_{k+1}(\beta)  =\pi \), yielding a symmetric configuration of order \( N = 2k \).

A similar argument applies for the case \( N = 2k + 1 \), completing the proof.
\end{proof}

\end{document}